\documentclass[10pt,a4paper]{amsart}

\usepackage{graphicx}
\usepackage{amssymb}
\usepackage{epstopdf  }
\usepackage{hyperref}
\usepackage[usenames]{color}

\usepackage{geometry}                

\input xy
\xyoption{all}
\usepackage{tikz}
\usetikzlibrary{arrows,shapes}

\renewcommand{\epsilon}{\varepsilon}
\renewcommand{\setminus}{\smallsetminus}
\renewcommand{\emptyset}{\varnothing}

\newtheorem{theorem}{Theorem}[section]
\newtheorem{proposition}[theorem]{Proposition}

\newtheorem{lemma}[theorem]{Lemma}

\theoremstyle{definition}
\newtheorem{example}[theorem]{Example}

\newtheorem{definition}[theorem]{Definition}

\newtheorem{remark}[theorem]{Remark}

\newcommand{\N}{\mathbb N}

\renewcommand{\L}{\mathcal L}

\newcommand{\cohom}[3]{H^{{\raise1pt\hbox{$\scriptstyle#1$}}}(#2\>\!,#3)}
\newcommand{\tatecohom}[3]%
  {\widehat H^{{\raise1pt\hbox{$\scriptstyle#1$}}}(#2\>\!,#3)}

\newcommand{\Cohom}[3]%
  {H^{{\raise1pt\hbox{$\scriptstyle#1$}}}\big(#2\>\!,#3\big)}
\newcommand{\Tatecohom}[3]%
  {\widehat H^{{\raise1pt\hbox{$\scriptstyle#1$}}}\big(#2\>\!,#3\big)}

\newcommand{\homol}[3]{H_{{\lower1pt\hbox{$\scriptstyle#1$}}}(#2\>\!,#3)}
\newcommand{\homolog}[2]{H_{{\lower1pt\hbox{$\scriptstyle#1$}}}(#2)}

\newcommand{\colim}{\varinjlim}

\newcommand{\EG}{{\operatorname{E}G}}
\newcommand{\ueg}{{\underline E}G}

\newcommand{\ov}[1]{\overline{#1}}

\title[Presentations of generalisations of Thompson's group $V$]{Presentations of generalisations of Thompson's group $V$.}

\author{C.~Mart\'inez-P\'erez}

\address{Conchita Mart\'inez-P\'erez, Departamento de Matem\'aticas, Universidad de Zaragoza,
50009 Zaragoza, Spain} \email{conmar@unizar.es}

\author{F. Matucci}

\address{
Francesco Matucci,
Instituto de Matem\'atica, Estat\'istica e Computa\c{c}\~ao Cient\'ifica, Universidade Estadual de Campinas, 13083-859, Campinas, Brazil}
\email{francesco@ime.unicamp.br}

\author{B.~Nucinkis}

\address{Brita Nucinkis, Department of Mathematics, Royal Holloway, University of London, Egham, TW20 0EX, UK.}\email{brita.nucinkis@rhul.ac.uk}

\date{\today} 

\thanks{This work was partially funded by an LMS Scheme 4 grant 41209. The first named author was supported by  Gobierno de Arag\'on, European Regional Development Funds and
 MTM2015-67781-P (MINECO/FEDER). 
The second author gratefully acknowledges the Fondation Math\'ematique Jacques Hadamard (ANR-10-CAMP-0151-02 - FMJH - Investissement d'Avenir) and the Funda\c{c}\~ao para a Ci\^encia e a Tecnologia  (FCT-PEst-OE/MAT/UI0143/2014) for the support received during the development of this work.}

\begin{document}

\begin{abstract} 
We consider generalisations of Thompson's group $V$, denoted by 
$V_r(\Sigma)$, which also include the groups of Higman, Stein and Brin. It was
shown in \cite{francescobritaconcha} that under some mild conditions these groups and centralisers
of their finite subgroups
are of type
$\mathrm{F}_\infty$. Under more general conditions we show
that the groups $V_r(\Sigma)$ are finitely generated and, under the mild conditions mentioned above, we see that they are finitely presented
and give a recipe to find explicit presentations.
For the centralisers of finite subgroups we
find a suitable infinite presentation and then apply a general procedure to shorten this presentation. In the appendix, we give a proof of this general shortening procedure.
\end{abstract}

\maketitle

\section{Introduction
\label{sec:introduction}
}

\noindent The original  Thompson groups $F\leq T \leq V$  are groups of homeomorphism of the unit interval, the circle and the Cantor-set respectively. In this note we consider generalisations of these groups, which are described as groups of automorphisms of certain Cantor algebras. These groups include  Higman's \cite{higman}, Stein's \cite{stein} and Brin's \cite{brin1} generalisations of $V.$

The groups $F$, $T$ and $V$ have attracted the attention of group theorists for several reasons, one of them is that there  are nice presentations and ways to represent elements available, making it possible to prove interesting results about metrics, geodesics and decision problems. However, the situation changes when one moves to some of their generalisations. There are presentations available for Higman's  groups  $V_{n,r}$  \cite{higman}, Stein's generalisations \cite{brinsquier, stein} and Brin's higher 
dimensional Thompson groups $sV$ \cite{hennigmatucci},
but  not for more complicated generalisations such as the groups $V_r(\Sigma)$ we are considering here. These were  defined by
the first and third authors together with Kochloukova \cite{desiconbrita2, britaconcha}  and were denoted $G_r(\Sigma).$  It is worth pointing out that elements in $V_r(\Sigma)$ admit a tree-pair representation similar to that of the original groups $F,T$ and $V$.
The authors show in \cite{francescobritaconcha} that, under some mild hypotheses, being
{\sl valid} and {\sl bounded}, $V_r(\Sigma)$ is the full automorphism group of a Cantor algebra. In the same paper it is shown that under some further minor restrictions, being {\sl complete}, these groups are of type $\mathrm{F}_\infty$ and that this also implies that  centralisers of finite subgroups are of type 
$\mathrm{F}_\infty$. We introduce all necessary background in Section \ref{sec:background}.
The structure of centralisers in $V_r(\Sigma)$ is  studied in detail in \cite{francescobritaconcha, britaconcha}.

One of the objectives of the present paper is to introduce a common framework to get explicit finite generating sets for the groups $V_r(\Sigma)$ in the case when the underlying Cantor-algebra $U_r(\Sigma)$ is  valid and bounded, and explicit presentations under the additional assumption that $U_r(\Sigma)$ is complete. To do that, we construct a model for $\operatorname{E}G$ for $G=V_r(\Sigma)$ and use this model to obtain explicit presentations of these groups.
  
In Section \ref{presentation} we also give an explicit finite presentation for centralisers of finite subgroups for those $V_r(\Sigma)$ 
that are finitely presented. To do so we use the so called 
{\sl Burnside procedure} as used 
by Guralnick, Kantor, Kassabov and Lubotzky
\cite{kassabovetc}.

In the Appendix we shall give an outline and proof of the  Burnside-procedure as used in \cite{kassabovetc}.  This procedure is well known, but  we are also not aware of any proofs elsewhere. The idea to use the procedure
is to look for an easy presentation for $G$, which is somehow symmetric and elementary,
but may have even infinitely many 
generators and relations.
The Burnside procedure offers a way to cut these down to a more manageable, and sometimes finite presentation.

\subsection*{Acknowledgments} We would like to thank 
Collin Bleak for helpful discussions related to \cite{matuccietc}
and Martin Kassabov for introducing to us  the Burnside procedure used in \cite{kassabovetc}.

\section{
\label{sec:background}
Background on generalised Thompson groups}

In this section we introduce those properties of valid bounded Cantor-algebras used in this paper. For detailed definitions and notation the reader is referred to \cite[Section 2]{britaconcha}, and for proofs of statements cited here, see \cite{desiconbrita2, francescobritaconcha, britaconcha}.

Let $S=\{1,\ldots,s\}$ be a finite set of colours
and associate to each $i \in S$ an integer  $n_i>1$, called arity of the colour $i.$ For every $i\in S$ consider  the following right operations on a set $U$: 
\begin{itemize}
\item[(i)] One $n_i$-ary operation $\lambda_i\, :\,U^{n_i}\to U,$ and 
\item[(ii)] $n_i$ 1-ary operations $\alpha^1_i,\ldots,\alpha^{n_i}_i\, ; \alpha^j_i:U\to U.$ 
\end{itemize}
 We also consider, for each $i\in S$ and 
 $v\in U,$ the map 
 $$\alpha_i:U\to U^{n_i}$$
  given by
$v\alpha_i:=(v\alpha^1_i,v\alpha^2_i,\ldots,v\alpha^{n_i}_i).$ The maps $\alpha_i$ are called {\sl descending} operations, or expansions, and the maps $\lambda_i$ are called {\sl ascending} operations, or contractions.

Fix a finite set $X_r$ of cardinality $|X_r|=r$. One can define the free objects on the set $X_r$ with respect to the previous operations. To define our generalisations of Thompson's group $V$,  we will be interested in the free object constructed under the extra requirement that a certain set of laws $\Sigma$ described below must be satisfied. This last free object is what we denote $U_r(\Sigma)$ and call the (free) Cantor-algebra on $X_r$ satisfying $\Sigma$.

\begin{definition}\label{sigmadef}\cite[Section 2]{britaconcha} Fix a finite set $X_r$ of cardinality $|X_r|=r$ and consider the operations i) and ii) above defined on a set $U$. Then  $\Sigma=\Sigma_1 \cup \Sigma_2$ with $\Sigma_1$ and $\Sigma_2$ the following set of laws:

\begin{itemize}

\item[i)] $\Sigma_1$ is given by
  $$u\alpha_i\lambda_i=u,$$
 $$(u_1,\ldots,u_{n_i})\lambda_i\alpha_i=(u_1,\ldots,u_{n_i}),$$
for every $u\in U$,  $i\in S$, and $n_i$-tuple: $(u_1,\ldots,u_{n_i})\in U^{n_i}.$

\item[ii)] $\Sigma_2$ is given by
$$\Sigma_2=\bigcup_{1\leq i<i'\leq s}\Sigma_2^{i,i'}$$
 where each $\Sigma_2^{i,i'}$ is either empty or  consists  of the following laws: consider first $i$ and fix a map $f:\{1,\ldots,n_{i}\}\to\{1,\ldots,s\}$. For each $1\leq j\leq n_{i}$, we see  $\alpha_{i}^j\alpha_{f(j)}$ as a set of length 2 sequences of descending operations and
let $\Lambda_{i}=\cup_{j=1}^{n_{i}}\alpha_{i}^j\alpha_{f(j)}$.  Do the same for $i'$ (with a corresponding map $f'$) to get $\Lambda_{i'}$.
We need to assume that $f,f'$ are chosen so that $|\Lambda_i|=|\Lambda_{i'}|$  and  fix a bijection $\phi:\Lambda_{i}\to\Lambda_{i'}$.
Then $\Sigma_2^{i,i'}$ is 
$$u\nu=u\phi(\nu)\quad \nu\in \Lambda_{i},u\in U.$$
\end{itemize}

\end{definition}

From now on we work with the free object $U_r(\Sigma)$ only.
Let $B\subset U_r(\Sigma)$, $b\in B$ and $i$ a colour of arity $n_i$. The set
$$(B\setminus\{b\})\cup\{b\alpha_i^1,\ldots,b\alpha_i^{n_i}\}$$
is called a {\sl simple expansion} of $B$. Analogously, if $b_1,\ldots,b_{n_i}\subseteq B$ are pairwise distinct,
$$(B\setminus\{b_1,\ldots,b_{n_i}\})\cup\{(b_1,\ldots,b_{n_i})\lambda_i\}$$
is a {\sl simple contraction} of $B$.
A finite chain of simple expansions (contractions) is an expansion (contraction). A  subset $A\subseteq U_r(\Sigma)$ is called {\sl admissible} if it can be obtained from the set $X_r$ by finitely many expansions or contractions. If a subset $A_1$ is obtained from a subset $A$ by an expansion (simple or not), then we write $A\leq A_1$.

\begin{remark}\label{adm-basis}
Recall  that $U_r(\Sigma)$ is said to be {\sl bounded} (see \cite[Definition 2.7]{britaconcha})  if for all admissible subsets $Y$ and $Z$ such that there is some admissible $A\leq Y,Z$, there is a unique least upper bound of $Y$ and $Z$. By a unique least upper bound we mean an admissible subset $T$ such that $Y \leq T$ and $Z \leq T$, and whenever there is an admissible set $S$ also satisfying $Y\leq S$ and $Z\leq S$, then $T \leq S.$

By \cite[Lemma 2.5]{desiconbrita2}, any admissible set is a basis of $U_r(\Sigma)$. 
Conversely, by \cite[Theorem 2.5]{francescobritaconcha},
if $\Sigma$ is valid and bounded, any basis of $U_r(\Sigma)$
is also an admissible 
set.
Furthermore, for every admissible subset of cardinality $m$, we have that 
$$m\equiv r \mbox{ mod }d \qquad \mbox{ for }d:=\text{gcd}\{n_i-1\mid i=1,\ldots,s\}.$$
In particular, any basis with $m$ elements can be transformed into one of $r$ elements. Hence $U_r(\Sigma)=U_m(\Sigma)$ and we may assume that $r \leq d.$ 
\end{remark}

\begin{definition}\label{groups}\cite[Definition 2.12]{britaconcha} Let $U_r(\Sigma)$ be a valid Cantor algebra. We denote the group of all Cantor algebra automorphisms of $U_r(\Sigma)$ by $V_r(\Sigma)$. In particular, these automorphisms are induced by a  map $V\to W$, where $V$ and $W$ are admissible subsets of $U_r(\Sigma)$ of the same cardinality. We denote the action of an automorphism on the left.
\end{definition}

\medskip\noindent For example, when $\Sigma_2=\emptyset$ and $s=1,$we retrieve the original Higman-Thompson groups $G_{r,n}$ \cite{higman}.
For $s=2$ and $r=1$, the Brin-Thompson groups are now given by the set $\Sigma_2$ that can be visualised as follows:

\medskip
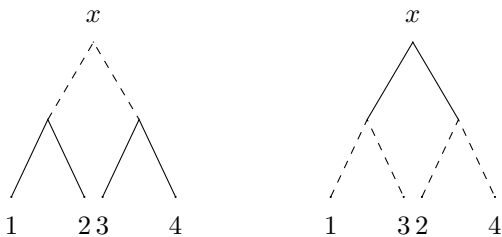
\begin{tikzpicture}[scale=0.6]

  \draw[black, dashed]
    (0,0) -- (1, 1.71) -- (2,0);
  \filldraw(1,1.71) circle (0.1pt) node[above=4pt]{$x$};

  \draw[black] (-0.8,-1.71) -- (0,0) --(0.8,-1.71);
  \draw[black] (1.2,-1.71) -- (2,0) --(2.8,-1.71);

      \filldraw (-0.8,-1.71) circle (0.3pt) node[below=4pt]{$1$};
      \filldraw (0.8,-1.71) circle (0.3pt) node[below=4pt]{$2$};
      \filldraw (1.2,-1.71) circle (0.3pt) node[below=4pt]{$3$};
      \filldraw (2.8,-1.71) circle (0.3pt) node[below=4pt]{$4$};

  \draw[black]
    (7,0) -- (8, 1.71) -- (9,0);
 \filldraw(8,1.71) circle (0.1pt) node[above=4pt]{$x$};

  \draw[black, dashed] (6.2,-1.71) -- (7,0) --(7.8,-1.71);
  \draw[black, dashed] (8.2,-1.71) -- (9,0) --(9.8,-1.71);

      \filldraw (6.2,-1.71) circle (0.3pt) node[below=4pt]{$1$};
      \filldraw (7.8,-1.71) circle (0.3pt) node[below=4pt]{$3$};
      \filldraw (8.2,-1.71) circle (0.3pt) node[below=4pt]{$2$};
      \filldraw (9.8,-1.71) circle (0.3pt) node[below=4pt]{$4$};

   \end{tikzpicture}

\noindent Here dotted and black lines represent expansions of different colours. For more examples the reader is referred to \cite{francescobritaconcha, britaconcha}.

\begin{remark}\label{basis}

\noindent If $U_r(\Sigma)$ is valid and bounded every element of $V_r(\Sigma)$ can be given by a bijection $V \to W$, where $V$ and $W$ are descendants of the fixed basis $X_r$.
\end{remark}

For $r=1$, this means that we can visualise elements of $V_1(\Sigma)$ by tree-pair diagrams of rooted trees, where the root represents the basis $X_1=\{x\}$. So, for example, when $s=1$ and $n=2$, $V_1(\Sigma)=V$, the original Thompson group, and the well-known generator $x_0\in F \subset V$ is described as follows:

\bigskip
\begin{tikzpicture}[scale=0.6]

  \draw[black]
    (0,0) -- (1, 1.71) -- (2,0);

  \draw[black] (1.2,-1.71) -- (2,0) --(2.8,-1.71);

      \filldraw (1.2,-1.71) circle (0.3pt) node[below=4pt]{$2$};
      \filldraw (2.8,-1.71) circle (0.3pt) node[below=4pt]{$3$};
      \filldraw (0,0) circle (0.3pt) node[below=4pt]{$1$};

\draw[black] (4.5, 0)   node{$\longrightarrow$};
\draw[black] (4.5, 0) node[above=4pt]{$x_0$};

  \draw[black]
    (7,0) -- (8, 1.71) -- (9,0);

  \draw[black] (6.2,-1.71) -- (7,0) --(7.8,-1.71);

      \filldraw (9,0) circle (0.3pt) node[below=4pt]{$3$};
          \filldraw (6.2,-1.71) circle (0.3pt) node[below=4pt]{$1$};
      \filldraw (7.8,-1.71) circle (0.3pt) node[below=4pt]{$2$};

   \end{tikzpicture}

\begin{definition}\cite[Definition 3.2]{francescobritaconcha} Let $B\leq A$ be admissible subsets of $U_r(\Sigma)$. We say that the expansion $B\leq A$ is {\sl elementary} if there are no repeated colours in the paths from elements in $B$ to their descendants in $A$.   We denote an elementary expansion by $B\preceq A.$  We say that the expansion is {\sl very elementary} if all paths have length at most 1. 
\end{definition}

\medskip\noindent Denote by  $\mathcal{P}_r$ the poset of admissible subsets in $U_r(\Sigma)$, and by $|\mathcal{P}_r|$ its geometric realisation.
We now describe  the Stein complex $\mathcal{S}_r(\Sigma)$ \cite{stein}, which is a subcomplex of $|\mathcal{P}_r|$. The vertices in $\mathcal{S}_r(\Sigma)$ are given by the admissible subsets
of $U_r(\Sigma)$. The  $k$-simplices are given by chains of  expansions $Y_0\leq\ldots\leq Y_k$, where $Y_0\preceq Y_k$ is an elementary expansion.

\medskip\noindent From now on we we will denote $V_r(\Sigma)$ by $G$.
It was shown in \cite{britaconcha} that $|\mathcal{P}_r|$ is a model for $\ueg$, the classifying space for proper actions. In the next section we will use
$\mathcal{S}_r(\Sigma)$ to construct a model for $\EG$.

\section{A model for $\text{E}G$}

\noindent In this section we construct a model for the space $\EG$ when $G$ is the automorphism group of a valid and bounded Cantor-algebra $U_r(\Sigma)$ as before. We shall use this model to get, initially infinite, presentations for our groups, which we will the reduce to obtain a finite generating set, and later a finite presentation under some extra hypothesis on $U_r(\Sigma).$

\subsection{Some technical observations} To begin with we collect a few 
technical observations that we will use later on. As seen before, 
the elements $g$ in our group $G$ can be expressed via a bijection 
between a pair of admissible subsets (or bases) $(B,B')$ both of the same cardinality. But observe 
that the pair above is not enough to determine $g$ and that
we have to specify the explicit bijection. A way to overcome this 
problem is to work with ordered bases, in the sense that instead of a 
basis $B$ viewed as a set, we will be considering an ordered tuple $A$ with 
underlying set $B$. We say $u(A)=B$ ($u$ for {\sl underlying}). A pair of ordered
tuples $(A,A'),$ both $A$ and $A'$ of the same cardinality, uniquely determines the 
element of $G$ mapping the elements of $A$ to the elements of $A'$ in the 
prescribed order and conversely, any group element is expressible in 
this way. Of course, just as for the representation of the
pair of bases, there is not a unique pair $(A,A')$ determining a 
given $g\in G$:  we 
may apply descending or ascending operations to $A$ and $A'$ in a 
consistent way to get a new pair of ordered tuples representing the same group element. Moreover, we may also permute the 
elements of both tuples  in a consistent way and still get the same 
$g$. This motivates the following definition:
given tuples $A_1$, $A_2$ with bases as underlying sets we put
$$A_1\precsim A_2\iff u(A_1)\leq u(A_2)\text{ is an elementary expansion.}$$
Equivalently, $A_1\precsim A_2$ if $u(A_1)\leq u(A_2)$ in the Stein poset.  
Abusing slightly the terminology, we will say in this case  that $A_2$ is obtained from $A_1$ by {\sl descending operations}. Essentially this means that we are considering the permutation of the elements of a tuple as a new type of descending operation. Of course this could equally be viewed as an \lq\lq ascending" operation, but it turns out to be convenient to view it as descending. If we want to record the precise operations that yield $A_2$ when applied to $A_1$ we will write
$$A_1\buildrel\epsilon\over\precsim A_2$$ 
and will also set $A_2=A_1\epsilon$. Observe that $\epsilon$ can be seen as a precise recipe to get $A_2$, and in $\epsilon$ it is encoded exactly which elements are modified, permuted and so on.

\subsection{The model for $\text{E}G$}\label{eg}

Let $Z$ be the complex constructed as follows: The  points of $Z$ are the ordered tuples $A$ with underlying set a basis $u(A)$ in the Stein complex $\mathcal{S}_r(\Sigma)$. For each chain
$$A_0\precsim\ldots\precsim A_k$$
we attach an (oriented) $k$-simplex at the
vertices $A_0,\ldots,A_k$. Observe that there might be repeated 
vertices, so this is not a simplicial complex but rather has the structure 
of a $\Delta$-complex, see \cite[Section 2.1]{hatcher}.
The group $G$ acts on the set of bases, and using that action one can 
define a $G$-action on $Z$ in the obvious way. Note that this action is free. In particular this 
implies that two different 1-simplices starting in $A_0$, say 
$A_0\precsim A_1$ and $A_0\precsim A_1'$ cannot be in the same 
$G$-orbit. Hence they yield different 1-simplices in the quotient complex 
$Z/G$. Conversely, if $\ov{A}_0\buildrel{\ov{\epsilon}}\over\to \ov{A}_1$ is an 
edge in $Z/G$, then once we have fixed a lift $A_0$ of $\ov{A}_0$ to 
$Z$, $\ov{\epsilon}$ lifts to a unique 1-simplex of $Z$. Therefore there 
is some well defined set of descending operations  yielding a tuple 
$A_1'$ which is uniquely determined so that 
$A_0\buildrel\epsilon\over\precsim A_1'$ is the lift of $\ov{\epsilon}$. Moreover, the tuple 
$A_1'$ is uniquely determined. Note that we have the extra 
restriction coming from the Stein poset: we can only apply  
descending operations of the same colour once to any element of $A_0$.

\noindent Applying the same argument implies that this also holds for any lift  of a path in $Z/G$ to $Z$.

\medskip\noindent We now show that $Z$ is contractible by using the contractibility of $\mathcal{S}_r(\Sigma).$
There is  a $G$-map 
$$u:Z\to\mathcal{S}_r(\Sigma)$$
associating the underlying basis to an ordered tuple.

Fix a basis $B\in\mathcal{S}_r(\Sigma)$  of cardinality $k$. Then $u^{-1}(B)$ is the full subcomplex of $Z$ with 0-simplices given by the 
tuples with underlying set $B$, i.e. given by all  possible permutations of the elements in $B$. Let $H$ the stabiliser of $B$ in $G$. Then $H$ is isomorphic to  the symmetric group of degree $k$ and acts freely on the 0-simplices of $u^{-1}(B)$. In fact we may choose a bijection between the 0-simplices of $u^{-1}(B)$ and the elements of $H$ and the definition of the complex structure of $Z$ means that any $(k+1)$-elements subset of 0-simplices spans a $k$-simplex. 

For example if $H=S_2$ is the symmetric group on two letters with elements $1$ and $x$, then the $1$-simplices are $\{1,1\}$, $\{1,x\}$, $\{x,1\}$ and $\{x,x\}$, and the $2$-simplices are $\{1,1,1\}$, $ \{1,1,x\}$ etc.

In other words, $u^{-1}(B)$ is easily seen to be the usual complex associated to the bar resolution of the finite group $H$, see for example \cite[Example 1B.7]{hatcher}. In particular this shows that  that $u^{-1}(B)$ is contractible.

Using  Quillen's Theorem A \cite{quillen}, we can now show  that $u$ is a homotopy equivalence. To see this, let $J_Z$  be the category with objects the simplices of $Z$ and morphisms given by the face relations. Note that since  $Z$ is not a simplicial complex, this is not a poset. Let $J_S$ be the poset of simplices in  $\mathcal{S}_r(\Sigma)$. The map $u$ induces a functor
$$J_u:J_Z\to J_S,$$
and the geometric realisations of nerves of the categories $J_Z$ and $J_S$ are the barycentric subdivisions of $Z$ and $\mathcal{S}_r(\Sigma)$ respectively. Once  we show that for any $\sigma:B_0<B_1<\ldots<B_t$ in $J_S,$
$$J_u/\sigma:=\{\tau\in J_Z\mid J_u(\tau)\text{ is a subsimplex of } \sigma\}$$
is a contractible subcategory of $J_Z,$  we can use Quillen's Theorem A to deduce that $J_u$  is a homotopy equivalence. The category $J_u/\sigma$ is just the category with objects the simplices in the join
$$u^{-1}(B_0)\star\ldots\star u^{-1}(B_t)$$
and morphisms given by face relations. As $u^{-1}(B_0)\star\ldots\star u^{-1}(B_t)$ is contractible, this category is also contractible.
Hence $J_u$ is a homotopy equivalence and thus $u$ is, too. Since $\mathcal{S}_r(\Sigma)$ is contractible we deduce that $Z$ is contractible as required.

\section{An infinite presentation}
\label{sec:infinite-presentation}

\noindent In this section we use the model for $\operatorname{E}G$ just constructed to obtain a presentation for our group. As the model is of infinite type, our presentation will initially be infinite, but we will see in the next section how it is possible in the case when the Cantor-algebra is valid and bounded  to reduce  the generating system to a finite one.

\noindent We obtain our presentation using the following well known result that we recall here for the reader's convenience. 

\begin{theorem}(\cite[Theorem 3.1.16 and Corollary 3.1.17]{geoghegan})\label{gpresentation} Let $G$ be a group and $Z$ a simply connected CW-complex with a free $G$-action such that $Z/G$ is oriented and path connected. Let  $\mathcal{T}$ be a maximal tree in $Z/G$. Let
\begin{itemize}
\item $W$ be the set of (oriented) 1-cells of $Z/G$,

\item $R$ be the set of words in the alphabet $W\cup W^{-1}$ obtained as follows: for each (oriented) 2-cell $e^2_\gamma$ in $Z/G$, let
 $\tau(e^2_\gamma)$ be a word representing the boundary $\delta e^2_\gamma$ and set
 $$R=\{\tau(e^2_\gamma)\mid e^2_\gamma\text{ oriented 2-cell of }Z/G\}.$$

\item $S\subset W$ be the set of (oriented) 1-cells of $\mathcal{T}$ (seen as one letter words in $W$).
\end{itemize}
Then
$$\langle W\mid R\cup S\rangle$$
is a presentation of the group $G \cong \pi_1(Z/G)$. If moreover $Z/G$ has finite 2-skeleton, then this is a finite presentation.
\end{theorem}

 \subsection{The isomorphism $G\cong \pi_1(Z/G)$.}\label{gensys1}
 We now give an explicit isomorphism between $G$ and the fundamental group of $Z/G$, where we are coming back to our previous notation so $G=V_r(\Sigma)$ and $Z$ is the complex of Subsection \ref{eg}. The standard way to do that is to fix some point $x_0\in Z$ and map the element  $g\in G$ to the loop in $Z/G$ obtained by taking the quotient of a path from  $x_0$ to $gx_0$ in $Z$. 
We shall take as $x_0$ a tuple with underlying set our preferred  basis of $r$ elements $X_r$. To ease notation, we denote this tuple  by $X_r$ as well.  As the $G$-action on $Z$ preserves the cardinality of each tuple, the 0-simplices of $Z/G$ correspond to the possible cardinalities of tuples (or of bases). By Remark \ref{adm-basis}, we recall that the possible cardinalities of the bases are exactly the integers  congruent to $r$ modulo $d$  where $n_1,\ldots,n_s$ are the arities and 
$$d=\text{gcd}(n_1-1,\ldots,n_s-1).$$
So  the 0-simplices of $Z/G$  can be labelled as
 $$\{\ov{X}_i\mid i\equiv r\text{ mod }d\}$$
 where the subindex is the cardinality of the associated bases. 
 Now, choose a maximal tree $\mathcal{T}$ in $Z/G$. The vertices of $\mathcal{T}$ are all the  0-simplices above and there is a unique path in $\mathcal{T}$ from $\ov{X}_r$ to each other $\ov{X}_{i}$. This path determines uniquely a precise tuple $X_i$ that is a lift of 
$\ov{X}_{i}$ (observe that  $X_i$ depends on the choice of $\mathcal{T}$).

 Let $\ov{X}_{i}\buildrel{\ov{\epsilon}}\over\to\ov{X}_{j}$ be an edge (thus $i\leq j$). By the comments above there is a uniquely determined lift  $X_{i}\buildrel{\epsilon}\over\to X_{j}'$ of $\ov{\epsilon}$, here $X_{j}'$ is a new tuple which  is in the same orbit as $X_{j}$. Therefore there is a uniquely determined $g\in G$ such that $X_{j}'=gX_{j},$ and this is precisely  the element in $G$ corresponding to the generator $\ov{\epsilon}\in\pi_1(Z/G)$. We have $g=(X_{j},X_{j}')$ and $X_{j}'=X_{i}\epsilon$.

\begin{example}\label{eltsV} Let $G$ be the original Thompson group $V$. In particular, $r=1, s=1, n=2.$ We can represent bases of $U_1(\Sigma)$ by finite rooted binary trees, and hence can choose $X_1$ to be a single point,  and  $X_2$ and $X_3$ to be the bases represented  as follows:

\medskip

 
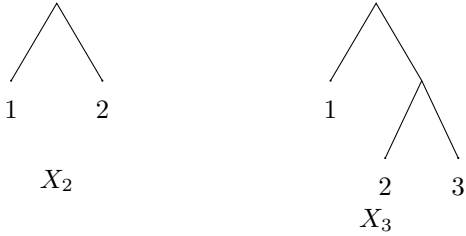
\begin{tikzpicture}[scale=0.6]
 \draw[black]
    (0,0) -- (1, 1.71) -- (2,0);

      \filldraw (2,0) circle (0.3pt) node[below=4pt]{$2$};
      \filldraw (0,0) circle (0.3pt) node[below=4pt]{$1$};

      \draw[black] (1,0)  node[below=30pt]{$X_2$};

  \draw[black]
    (7,0) -- (8, 1.71) -- (9,0);

  \draw[black] (8.2,-1.71) -- (9,0) --(9.8,-1.71);

      \filldraw (7,0) circle (0.3pt) node[below=4pt]{$1$};
          \filldraw (8.2,-1.71) circle (0.3pt) node[below=4pt]{$2$};
      \filldraw (9.8,-1.71) circle (0.3pt) node[below=4pt]{$3$};

 \draw[black] (8,0)  node[below=45pt]{$X_3$};

\end{tikzpicture}

Suppose we take $\varepsilon$ to be the expansion of $X_2$ on the left-hand leaf. This gives us  $X_3'$ as follows:
\medskip
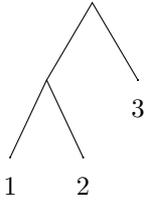
\begin{tikzpicture}[scale=0.6]

  \draw[black]
    (0,0) -- (1, 1.71) -- (2,0);

  \draw[black] (-0.8,-1.71) -- (0,0) --(0.8,-1.71);

      \filldraw (-0.8,-1.71) circle (0.3pt) node[below=4pt]{$1$};
      \filldraw (0.8,-1.71) circle (0.3pt) node[below=4pt]{$2$};
      \filldraw (2,0) circle (0.3pt) node[below=4pt]{$3$};

\end{tikzpicture}

\medskip\noindent and the corresponding element of $V$ is $x_0$ as described after Example \ref{basis}.

\end{example}

 \subsection{The maximal tree $ \mathcal{T}$.}\label{gensys2} To be able to write down an explicit presentation, the choice for $\mathcal{T}$ becomes important. This relies heavily on the choice of representative for $\ov{X_i}$ above.  This amounts to choosing a particular set of bases $X_i$ in $U_r(\Sigma),$ where $i\equiv r \text{ mod } d.$

 \begin{example}\label{treeV} For $G=sV$ we have $r=1$ and for each $k \in \N$ there is a basis $X_k.$ Again, these can be represented by finite rooted binary trees. Now fix a colour $i\in S$ and choose the $X_k$ as follows: we begin with $X_1$  our fixed one-element basis represented by a single point. Now $X_2$ is the basis obtained by applying the descending operation of colour $i$ to $X_1.$ We successively chose $X_k$ as obtained from $X_{k-1}$ by applying  the descending operation of colour $i$ to the last element of $X_{k-1}$ and labeling the elements in successive order. The representation for $X_k$ by a binary tree then looks as follows:
 
 \medskip
 
  
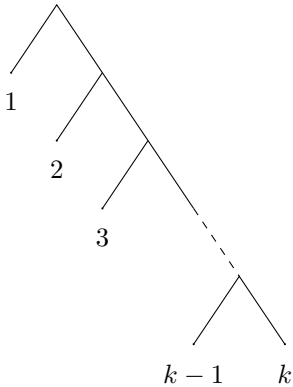
\begin{tikzpicture}[scale=0.6]
  \draw[black] (0,6) -- (1,7.5) -- (4,3);
   \draw[black] (1,4.5) -- (2,6);
   \draw[black] (2,3) -- (3,4.5);
 \draw[black] (4,0) -- (5,1.5) -- (6,0);
 \draw[black,dashed] (4,3) -- (5,1.5);

  \filldraw (0,6) circle (0.3pt) node[below=4pt]{1};
   \filldraw (1, 4.5) circle (0.3pt) node[below=4pt]{2};
    \filldraw (2,3) circle (0.3pt) node[below=4pt]{3};
 \filldraw (4,0) circle (0.3pt) node[below=4pt]{$k-1$};
  \filldraw (6,0) circle (0.3pt) node[below=4pt]{$k$};
\end{tikzpicture}

\medskip
\noindent Notice that for $V_1(\Sigma)$ we can always choose the $X_k$ to be represented by a rightmost tree as above, provided that  all colours have the same arity. Now the construction shows that the maximal tree $\mathcal{T}$ in $Z/G$ is a rooted infinite line.
\end{example}

\noindent For example, the baker's map $b \in 2V$ can be easily described using the bases chosen above. Let $X_1$ be a single point and $X_2$ be 
as in Example \ref{eltsV}; note that we expanded with colour $1$. Now we consider $X_2'$ the basis obtained from $X_1$ by expanding once with colour $2$ (represented by a dashed line). Hence this gives rise to the element $b \in 2V.$

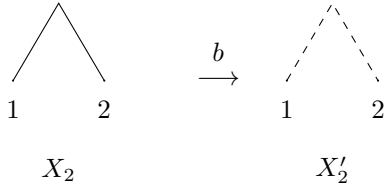
\medskip

 \begin{tikzpicture}[scale=0.6]
 \draw[black]
    (0,0) -- (1, 1.71) -- (2,0);

      \filldraw (2,0) circle (0.3pt) node[below=4pt]{$2$};
      \filldraw (0,0) circle (0.3pt) node[below=4pt]{$1$};

      \draw[black] (1,0)  node[below=26pt]{$X_2$};

     \draw[black, dashed]
    (6,0) -- (7, 1.71) -- (8,0);

  \draw[black] (4.5, 0)   node{$\longrightarrow$};
\draw[black] (4.5, 0) node[above=4pt]{$b$};

      \filldraw (6,0) circle (0.3pt) node[below=4pt]{$1$};
      \filldraw (8,0) circle (0.3pt) node[below=4pt]{$2$};

      \draw[black] (7,0)  node[below=25pt]{$X'_2$};

\end{tikzpicture}

\medskip\noindent For the general case with mixed arities we will not be able to find such a straightforward set of representatives $X_i$ as before. We will show that we can, however, find a maximal tree $\mathcal{T}$ in $Z/G$, whose vertices are all but a finite number obtained by a step-by step process beginning with our fixed basis $X_r$ and then  expanding the last element of a basis previously constructed.

\begin{example}\label{mixed arities} Let $V_r(\Sigma)$ be the group given by $r=1$, $s=2$, $n_1=5$ and $n_2=7.$ Then $d=2$ and  our chosen set of bases is of the form
$$\{X_i \, |\, i \equiv 1 \mbox{ mod } 2\}.$$
By simply expanding $X_1$ by the colours $1$ and $2$ respectively, we obtain $X_5$ and $X_7$. To  obtain a $X_3$ we could contract the last $5$ elements of $X_7$ by colour $1$, but there is no way to obtain $X_3$ from $X_1$ by simply expanding.
\end{example}

\begin{remark}\label{max-tree}
We now describe the construction of our preferred maximal tree $\mathcal{T}$ in $Z/G$, where $G=V_r(\Sigma)$ is the automorphism group of a valid and bounded Cantor algebra. We begin by showing that we can obtain all but finitely many of the bases
$$\{X_i \,|\, i \equiv r \mbox{ mod } d\}$$
from $X_r$ applying descending operations only. 
In other words
$$\{r+\sum_{i=1}^s k_i(n_i-1)\mid 0\leq k_1,\ldots,k_s\}\cup P=\{r+kd\mid 0\leq k\},$$
where $P$ is a finite set of integers. To see this, observe first that the problem can be reduced to the case when $r=0$ and $d=1$. Now choose integers $k_1,\ldots,k_s$ such that 
$$1=\sum_{i=1}^s k_i(n_i-1)$$
and use them to produce integers $m_1,...,m_s$ with  $0\leq m_2,\ldots,m_s$ such that
$$1=\sum_{i=1}^s m_i(n_i-1).$$
Hence
$$1\equiv \sum_{i=2}^s m_i(n_i-1)   \mbox{ mod } m_1.$$
Multiplying this expression by the integers $2,\ldots,m_1$ we get positive numbers $a_1,\ldots,a_{m_1}$ which are a complete set of representatives of the residues modulo $m_1$ and such that they all belong to $\sum_{i=1}^s\N(n_i-1)$. Now, let $m$ be any integer with $m\geq \mbox{max}\{a_i\,|\, 1\leq i\leq m_1\}.$ Then for some such $i$, we have $m\equiv a_i$ modulo $m_1$ thus $m-a_i=lm_1$ for some $l\geq 0$. From this we deduce that  $m$ also belongs to $\sum_{i=1}^s\N(n_i-1)$.

It is now easy to find a (non maximal) directed tree in $Z/G$ having $\ov{X}_r$ as root and such that the cardinalities of the vertices are precisely the set
$r+\sum_{i=1}^s\N(n_i-1)$. Moreover, we can do it in such a way that the descending operations are always applied to the last element of each tuple. 
There are only finitely many points of our space $Z/G$ not in this tree. Choose one of them and consider a path from that point to some point of the tree Adding this path we get a new tree which is no longer rooted. If there are still points left then repeat the process. Eventually, we get a tree with the desired properties and having only finitely many roots.
\end{remark}

  \subsection{The presentation}\label{gensys3} Now we apply Theorem \ref{gpresentation} to produce an explicit presentation. We do get an abstract group presentation but we can also write it down as a presentation in terms of elements given by pairs of ordered bases using the explicit isomorphism in Subsection \ref{gensys1}, which allows one to recognise the group elements in a much more familiar way.  
  Recall that we have fixed a set of tuples 
 $$\{X_i\mid i\equiv r\text{ mod }d\}$$
 which are lifts of the nodes of our tree $\mathcal{T}$. Moreover there is a tree in $Z$ that 
 is a lift of $\mathcal{T}$.

By \cite[Theorem 3.1.16]{geoghegan}, $\pi_1(Z/G)$ is  generated by the edges in $Z/G$, i.e. by the 1-simplices $\ov{X}_{i}\buildrel{\ov{\epsilon}}\over\to\ov{X}_{j}$ in $Z/G$. As we have seen before, these correspond to elements $g\in G$ which are given by pairs $(X_{j}, X_{i}\epsilon)$
where $\epsilon$ is a set of descending operations. 

There are two sets of relators:

\begin{enumerate}
\item Relators of the form $\ov{\epsilon}=1$ whenever $\ov{\epsilon}$ is an edge in the tree $\mathcal{T}$. This means that there are tuples $X_i$ and $X_j$ in $\mathcal{T}$ such that $X_j$ is obtained from $X_i$ performing the operations $\epsilon$. The group element that corresponds to $\epsilon$ is $(X_j,X_j)$.  
\item Relators obtained from the boundaries of the 2-cells in $Z/G$. The 2-cells of $Z/G$ come from 2-cells in $Z$ and these are of the form $A_0\precsim A_1\precsim A_2$. Let $\epsilon_1$ be the set of operations needed to obtain $A_1$ from $A_0$, $\epsilon_2$ the set of operations needed to obtain $A_2$ from $A_1$ and $\epsilon$ the composition of $\epsilon_1$ and $\epsilon_2$. Passing down to the quotient $Z/G$ we get a 2-cell with boundary labelled $\ov{\epsilon}_1$, 
$\ov{\epsilon}_2$ and $\ov{\epsilon}$. So we have the relator
$$
\ov{\epsilon}=\ov{\epsilon}_1\ov{\epsilon}_2.$$
All this means  that this second set of relators consists of the ``composition of paths''. We want to write down this in terms of pairs of ordered bases. Let $i$ be the cardinality of $A_0$ and $j_1$ the cardinality of $A_1$. The edge $\ov{\epsilon}_1$ represents the element $g_1=(X_{j_1},X_i\epsilon_1)\in G$. We may apply the descending operations $\epsilon_2$ to this pair and then we observe that also $g_1=(X_{j_1}\epsilon_2,X_i\epsilon_1\epsilon_2).$ Note here that this follows from the definition of tree pair representation, and we do not need anything about the presentation that we are building. 
    Let $j_2$ be the cardinality of $A_2$, then $\ov{\epsilon}_2$ represents the element $g_2=(X_{j_2},X_{j_1}\epsilon_2)$ and $\ov{\epsilon}$ represents $g=(X_{j_2},X_i\epsilon)$. So we get the relator $g=g_1g_2$. In the particular case when $X_{j_1}\epsilon_2$  belongs to the lift of our tree $\mathcal{T},$ or equivalently when $\ov{\epsilon}_2$ belongs to $\mathcal{T}$, there is also a relator $g_2=1$ and we deduce $g=g_1$. This can also be seen using tree pairs:   as $X_{j_2}$ belongs to the prefixed set of nodes and has the same cardinality as $X_{j_1}\epsilon_2$, we must have $X_{j_1}\epsilon_2=X_{j_2}$.
\end{enumerate}
    
We may summarise as follows:

$$G=\langle W\mid R\rangle$$
where 
 $$W=\{(X_{j}, X_{i}\epsilon)\mid\epsilon \text{ is a sequence of descending operations and }X_i\neq X_j\}\,\text{ and}$$
$$R=\{g=g_2g_1\mid g=(X_{j_2},X_{j_1}\epsilon_2),\, g_1=(X_{j_1},X_i\epsilon_1),\, g_2=(X_{j_2},X_{j_1}\epsilon_2),\, \text{and }\epsilon=\epsilon_1\epsilon_2\}.$$

\noindent Alternatively, we may delete the pairs $(X_{j}, X_{i}\epsilon)$ such that $\ov{\epsilon}$ lies in the tree $\mathcal{T}$ from our list of generators.


\subsection{Reducing the generating set}\label{gensys4}  A quick look to the generating set we have just obtained shows that it is far too big. Reducing it can be a complicated task but there is a reduction that seems natural: our generators come from edges in $Z/G$ and these edges come from descending operations, so one expects that edges coming from \lq\lq very elementary operations" should be enough. 
This is in fact the case but to make it more precise we need now some additional technicalities. Let us fix what should be called \lq\lq very elementary" in our context.  An edge $A_1\buildrel\epsilon\over\precsim A_2$ in $Z$  is {\sl very elementary} if it consists of a single operation, i.e., if it is either a permutation or it is a single descending operation (in this case $u(A_1)<u(A_2)$ is very elementary) but we do not allow composition of both. The case when $u(A_1)<u(A_2)$ will be termed {\sl strict} and for these type of operations we will assume that if the original tuple is
$(x_1,\ldots,x_i)$ and we apply the descending operation $\alpha$ at the $k$-th element then the resulting tuple is
$$(x_1,\ldots,x_{k-1},x_k\alpha^1,\ldots,x_k\alpha^{n_\alpha},x_{k+1},\ldots,x_i).$$
Any $\epsilon$ can be written as a composition of very elementary operations. 
Of course it may happen that different sequences of operations give the same result when applied to the same tuple. The following happens in the situations below, which we shall refer to as {\sl moves}:

\begin{itemize}
\item[i)] Disjoint type: we may apply two very elementary strict descending operations acting on distinct elements of a tuple and we get the same regardless of the order of application of these
two operations.

\item[ii)] $\Sigma$ type: we have different chains of elementary strict descending operations such that, up to a permutation, give the same result when applied to any element of any tuple and which come from the defining relations for the algebra encoded in $\Sigma$.

\item[iii)] Permutation-descending: we may first permute the elements of a tuple and then apply a very elementary strict descending operation or do it the other way around in a consistent manner and get the same result.

\item[iv)] Permutation: the composition of two permutations is still a permutation.
\end{itemize}

\begin{lemma}\label{moves} Let $A_1,A_2$ be tuples. If two different chains of very elementary descending operations yield $A_2$ when applied to $A_1$, then one can be obtained from the other by repeated application of moves of the four types above. 
\end{lemma}
\begin{proof} By doing moves of types iii) and iv) only we may assume that our two chains are of the form
$$\epsilon_1\epsilon_2\ldots\epsilon_t\sigma,$$
$$\epsilon'_1\epsilon'_2\ldots\epsilon'_{t'}\sigma',$$
where all  $\epsilon_i$, $\epsilon_i'$ are very elementary and strict and $\sigma$, $\sigma'$ are permutations. 
Consider first what happens when we look at the underlying sets $u(A_1)$ and $u(A_2)$. The fact that both series of operations give the same set when applied to $u(A_1)$, implies that, for each particular element, we are either performing the same operation or the same operation up to applying some of the relators encoded in $\Sigma$. This means that $\epsilon_1\epsilon_2\ldots\epsilon_t$ can be transformed to $\epsilon'_1\epsilon'_2\ldots\epsilon'_{t'}$ by doing moves of types i) or ii)
without taking the order of the elements into account. The fact that the relations in $\Sigma$ involve certain permutations implies that what we really get is that via some extra moves of types iii) and iv), $\epsilon_1\epsilon_2\ldots\epsilon_t$ is transformed to $\epsilon'_1\epsilon'_2\ldots\epsilon'_{t'}\tau$ for a certain permutation $\tau$. So at this point our two sequences are
$$\epsilon'_1\epsilon'_2\ldots\epsilon'_{t'}\tau\sigma,$$
$$\epsilon'_1\epsilon'_2\ldots\epsilon'_{t'}\sigma'.$$
The fact that both sequences yield $A_2$ when applied to $A_1$ implies that $B\tau\sigma=B\sigma'$ for $B=A_1\epsilon'_1\epsilon'_2\ldots\epsilon'_{t'}$, which is a move of type iv). 
\end{proof}

We next use Tietze transformations next to change the presentation above. Essentially, what we 
need to do is the following: whenever there is a relator $g=g_1g_2$ we
delete $g$ from our set of generators.
The effect of this transformation on the generating set is that we no longer have elements $g$ coming from edges which are not very elementary. Moreover we will have only two kinds of generators: strict generators coming from strict very elementary edges,  and finite order generators coming from permutations. We denote these sets by
$$W_s=\{(X_j,X_i\epsilon)\mid\epsilon\text{ a very elementary strict expansion},j=|X_i\epsilon|\}$$
and call these very {\sl elementary strict generators}. We also
consider the elements of the set
 $$W_p=\{(X_i,X_i\sigma)\mid\sigma\text{ a permutation}\},$$
and call them {\sl permutations}. From now on we will  use the term {\sl strict generators} for elements in $W_s$ instead of the more precise very elementary strict generator.

The effect of this transformation on the set of relators is as follows: we no longer have to consider relators coming from edges in the tree. Whenever there are two sequences of very elementary operations that give the same $A_2$ when applied to some $A_1$, we have a new relator. Lemma  \ref{moves} implies that these relators can be obtained from relators of the following types
\begin{itemize}
 \item[i)] $R_D$ contains relators of the form $g_1g_2=g_2'g_1'$ with $g_1,g_2,g_1',g_2'$  strict generators coming from moves of disjoint type.
 
 \item[ii)] $R_\Sigma$ contains relators between  strict generators possibly followed by a permutation coming from moves of $\Sigma$ type.
 
 \item[iii)] $R_{PD}$ contains relators of the form $g\sigma=\sigma g$ with $g$ a strict generator and $\sigma$ a permutation coming from moves of type iii).
 
 \item[iv)] $R_{P}$ contains relators of the form $\sigma=\sigma_1\sigma_2$ with $\sigma$, $\sigma_1$ and $\sigma_2$ permutations coming from moves of type iv).
 \end{itemize}
 
 Thus $G$ admits the following (infinite) presentation
 
 \begin{equation}\label{infinitepresentation}\langle W_s\cup W_p\mid R_D\cup R_{\Sigma}\cup R_{PD}\cup R_{P}\rangle.\end{equation}
 

\subsection{Being more explicit}\label{gensys5} 
Let us consider an arbitrary strict generator $(X_j,X_i\epsilon)$  associated to the strict edge $\ov{\epsilon}$. It is completely determined by a triple
$(i,k,t)$ meaning that $\ov{\epsilon}$ is obtained by applying the descending operation of colour $t$ to the $k$-th element of an orbit representative of the set of tuples of order $i$. We will use the triple to denote the generator.
Now we are going to write down explicitly how relators of disjoint type look like with this 
new notation. Recall that these relators come from very elementary strict descending and disjoint operations $\epsilon_1,\epsilon_2$ on  one hand, and $\epsilon_2',\epsilon_1'$ on the other. They are such that
$$\epsilon_1\epsilon_2=\epsilon_2'\epsilon_1'$$
where $\epsilon_1$ and $\epsilon_1'$ are operations of the same colour, say $t$, whereas $\epsilon_2$ and $\epsilon_2'$ are of colour $s$. Moreover $\epsilon_1$ acts at the $k_1$-th and $\epsilon_2'$ acts at the $k_2$-th elements of $X_i$. We may assume that $k_1<k_2$. Observe that this means that if we apply a descending operation to the $k_2$-th element first then the $k_1$-th element remains the same, but if we do it the other way around, i.e., apply a descending operation of colour $t$ to the $k_1$-th element first, then the former $k_2$-th element becomes the $(k_2+n_{t}-1)$-th. Therefore the triples associated to each of $\ov{\epsilon}_1,\ov{\epsilon}_2,
\ov{\epsilon}_2',\ov{\epsilon}_1'$ are
 $$\begin{aligned}
 \ov{\epsilon}_1&:\, (i,k_1,t)=(X_{i+n_{t}-1},X_i\epsilon_1),\\
  \ov{\epsilon}_2&:\, (i+n_{t}-1,k_2+n_{t}-1,s)=(X_{i+n_{t}-1+n_{s}-1},X_{i+n_{t}-1}\epsilon_2),\\
 \ov{\epsilon}_2'&:\, (i,k_2,s)=(X_{i+n_{s}-1},X_i\epsilon_2'),\\
\ov{\epsilon}_1'&:\, (i+n_{s}-1,k_1,t)=(X_{i+n_{s}-1+n_{t}-1},X_{i+n_{s}-1}\epsilon_1),\\
 \end{aligned}$$ 
 and our relator is
 \begin{equation}\label{genericD}(i,k_1,t)(i+n_{t}-1,k_2+n_{t}-1,s)=(i,k_2,s)(i+n_{s}-1,k_1,t)\end{equation}

 Analogously, it is possible to represent a generator $(X_i,\sigma(X_i))$ of \lq\lq permutation type" using the pair $(i,\sigma)$. Now, relators of type $R_{PD}$ come from the fact that applying first a permutation and then a very elementary  strict operation to a tuple, yields  the same as doing it the other way around for a suitable permutation. More explicitly, assume that we start with the tuple $X_i$. Let $\epsilon$ be the operation associated to the triple, say, $(i,k,t)$ and consider a permutation $\sigma$ represented by the pair $(i,\sigma)$. Slightly abusing notation view $\sigma$ as a permutation of the numbers $\{1,\ldots,i\}$. Starting with $X_i$ and performing first the permutation $\sigma$ and then applying the strict descending operation associated to $\ov{\epsilon}'=(i,\sigma(k),t),$ yields the tuple $X_i\sigma\epsilon'$ whose underlying set is the same as that of the tuple $X_i\epsilon$. Therefore there is some permutation $\sigma'$ such that the tuples $X_i\sigma\epsilon'$ and $X_i\epsilon\sigma'$ coincide. And this implies that we have a relator $\ov{\sigma}\cdot \ov{\epsilon}'=
 \ov{\epsilon}\cdot \ov{\sigma}'$ or
\begin{equation}\label{genericPD}(i,\sigma)(i,\sigma(k),t)=(i,k,t)(i+n_t-1,\sigma').\end{equation}

\section{A finite generating set
\label{sec:finite-generation}
}

 In this section, we  show that the generating system $W_s\cup W_p$ can be reduced to a finite one. We begin with $W_s$.
 We will use the following two particular cases of relators of disjoint type:
 
 \smallskip
 
 \noindent \emph{Case 1}: Let $(i,k,t)$ be a triple such that $i-k>n_l-1$ for any colour $l$ where we include the case $l=t$. Assume moreover that the terminal point of the associated edge in $Z/G$, i.e. $\ov{X}_{i+n_t-1}$, is not a root of the tree $\mathcal{T}$. Recall that this edge consists of applying a descending operation of colour $t$, which increases the cardinality in $n_t-1$. Then there is some edge of $\mathcal{T}$ ending in $\ov{X}_{i+n_t-1}$. Let $s$ be the colour of this last edge which is represented as  a triple by $(i+n_t-n_s,i+n_t-n_s,s)$ (recall that we constructed the tree $\mathcal{T}$ in such a way that the last element of each tuple is always being expanded). Now, as $i-k>n_s-1$ we deduce $k<i-n_s+1.$ Thus there is a relator of disjoint type as (\ref{genericD}) but with $i-n_s+1$ instead of $i$, $k$ instead of $k_1$ and $i-n_s+1$ instead of $k_2$. This relator is
 $$(i-n_s+1,k,t)(i+n_t-n_s,i+n_t-n_s,s)=(i-n_s+1,i-n_s+1,s)(i,k,t).$$
Since there is also a relator
 $$(i+n_t-n_s,i+n_t-n_s,s)=1,$$
because it belongs to $\mathcal{T}$, we deduce
 \begin{equation}\label{reducing1}
 (i,k,t)=(i-n_s+1,i-n_s+1,s)^{-1}(i-n_s+1,k,t).
 \end{equation}
This means that $(i,k,t)$ can be expressed in terms of triples with a smaller value of $i$.
 \smallskip
 
 \noindent \emph{Case 2}:  Let $(i,k,t)$ be a triple such that $i\geq k\geq n_t+1$.  Then $k-n_t+1> 1$ and $i-n_t+1\geq 2$. This means that there  is a relator of disjoint type as in (\ref{genericD}) but with $i-n_t+1$ instead of $i$, $1$ instead of $k_1$ and $k-n_t+1$ instead of $k_2$. This relator is
 $$(i-n_t+1,1,t)(i,k,t)=(i-n_t+1,k-n_t+1,t)(i,1,t).$$
From this we deduce 
\begin{equation}\label{reducing2}
(i,k,t)=(i-n_t+1,1,t)^{-1}(i-n_t+1,k-n_t+1,t)(i,1,t),
 \end{equation}
 meaning that $(i,k,t)$ can be expressed in terms of triples with either a smaller value of $i$ or with $k=1$.
 
 Observe now that arguing by induction on $i+k$, equations (\ref{reducing1}) and (\ref{reducing2}) imply  that any element in $W_s$ lies in the finite subgroup generated by the finite subset
 $$\{g\in W_s\mid\text{ the associated triple fails to fulfill both the conditions in case 1 and in case 2}\}.$$

 \begin{example}\label{infgenV} Let us consider  the group $V$, i.e. here we have one colour $t$ and $n_t=2.$ For now let us only concentrate on the strict generators $W_s.$ Note that an element $(i,i,t)$ is the identity. Looking at the representation by tree-pair diagrams, and the choice of $X_i$ in Example \ref{treeV},  it implies that we expand the rightmost leaf of the rightmost tree $X_i$, hence we obtain $X_{i+1}$ and the group element is represented by $(X_{i+1},X_{i+1})$, which is the identity.
Now consider elements $(i,k,t)$, where $k<i-1.$ Again, using the rightmost-tree, we see that after deleting unnecessary carets on the right, we get 
$$(i,k,t)=(k+1,k,t),$$
which is exactly the relator (\ref{reducing1}). For example, consider $(3,1,t)$. Then the corresponding tree-pair diagram looks as follows:

\bigskip
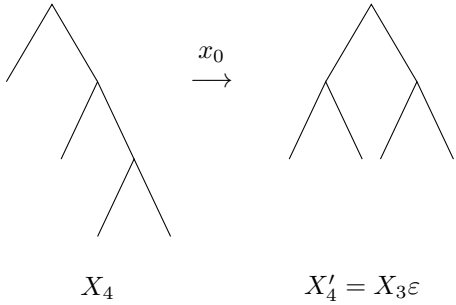
\begin{tikzpicture}[scale=0.6]

  \draw[black]
    (0,0) -- (1, 1.71) -- (2,0);

  \draw[black] (1.2,-1.71) -- (2,0) --(2.8,-1.71);
  
   \draw[black] (2.0,-3.42) -- (2.8,-1.71) --(3.6,-3.42);

          \filldraw (2,0)  node[below=70pt]{$X_4$};

\draw[black] (4.5, 0)   node{$\longrightarrow$};
\draw[black] (4.5, 0) node[above=4pt]{$x_0$};

  \draw[black]
    (7,0) -- (8, 1.71) -- (9,0);

  \draw[black] (6.2,-1.71) -- (7,0) --(7.8,-1.71);
    \draw[black] (8.2,-1.71) -- (9,0) --(9.8,-1.71);

      \filldraw (7.8,-1.71)  node[below=40pt]{$X'_4=X_3\varepsilon$};

   \end{tikzpicture}

   \medskip\noindent In particular, after deleting the rightmost caret in each tree, this is exactly the element $x_0$, see the picture after Remark \ref{basis}.
   
   \noindent Writing
   $$x_{i-2} = (i,i-1,t),$$
   we recover the well  known infinite generating set $\{x_k \,|\, k\geq 0\}$ for $F <V.$ Furthermore, this enables us to simplify the relator (\ref{genericD}) above. We have
   $$(i, k_1,t)(i+1, k_2+1,t)=(i, k_2,t)(i+1, k_1,t).$$
   Using that $(i,k,t)=(k+1,k,t)$ for $k<i-1$, we get the well known relator
   
$$x_k^{-1}x_lx_k=x_{l+1}$$
for any $k$ and $l$. Moreover, observe that strict generators and disjoint relators give us the well known infinite presentation of Thompson's group $F$ (see \cite{cfp}).

 \end{example}

Now we want to reduce $W_p$ in a similar way. The most natural way to do that is using relators of type $R_{PD}$, i.e. those mixing permutations and strict generators. To be able to argue by induction as before, we need to show  that if $i$ is big enough,  any element of the form $(X_i,\sigma(X_i))$, where $\sigma$ is a permutation, can be expressed in terms of permutations with a smaller $i$ and possibly strict generators. As the group of permutations of the tuple $X_i$ is generated by transpositions, we may assume that $\sigma$ itself is a transposition. Now, assume that $i\geq 3n_t$ for $t$ a colour with smallest possible arity $n_t$. As $\sigma$ only moves 2 elements, we may find $n_t$ consecutive elements in $X_i$ which are untouched by $\sigma$. Let $k$ be such that the $k$-th element in $X_i$ is the first one of those $n_t$ consecutive elements, and consider the strict generator associated to the triple $(i-n_t+1,k,t)$. Let $\sigma'$ be the transposition of $X_{i-n_t+1}$ that moves precisely the elements that are also moved by $\sigma$. Then the associated relator (\ref{genericPD}) with $i-n_t+1$ instead of $i$, and $\sigma$, $\sigma'$ interchanged is 
$$(i-n_t+1,\sigma')(i-n_t+1,k,t)=(i-n_t+1,k,t)(i,\sigma).$$
Thus 
\begin{equation}\label{reducing3}(i,\sigma)=(i-n_t+1,k,t)^{-1}(i-n_t+1,\sigma')(i-n_t+1,k,t)\end{equation}  
 as we wanted to show.
  
This discussion can be summarised as follows:

\begin{theorem}\label{finitegeneration} Assume that $U_r(\Sigma)$ is valid and bounded. Then $V_r(\Sigma)$ is generated by the finite set consisting of elements of the following three types:
\begin{itemize}
\item[1)] Strict generators associated to triples $(i,k,t)$ with $i\leq n_t+1$ and $i-k\leq n_s$ for any colour $s$.

\item[2)] Strict generators associated to triples $(i,k,t)$ such that $\ov{X}_{i+n_t-1}$ is a root of the tree $\mathcal{T}$.

\item[3)] Permutations associated to pairs $(i,\sigma)$ such that $i<3n_t$ for some colour $t$.
\end{itemize}
 \end{theorem}

\begin{example}\label{generatingV}
Consider $G=V.$ In Example \ref{infgenV}, we have already recovered the infinite presentation for $F < V.$ In the tree of Example \ref{treeV}, the triples have a single root $X_1$ so we do no have to consider generators as in item 2) of Theorem \ref{finitegeneration}. 
As before, let $i\geq 2$ and denote by $x_{i-2}$ the group element associated to triple $(i,i-1,t)$. Then from Theorem \ref{finitegeneration} one deduces the well known fact that the elements $x_i$, $i\geq 1$ together with the permutations generate the group and that  $x_0$ and $x_1$  
plus permutations are enough.
\end{example}

\begin{remark} Similar generating systems can be obtained without using the space $Z$ by proceeding as  Burillo and Cleary did for the Brin-Thompson groups $sV$ \cite[Theorem 2.1]{burillocleary}. Instead of our first step (Subsection \ref{gensys1}), fix a set of tuples, one for each possible cardinality, which are to be the \lq\lq source tree" of our tree pairs, and as \lq\lq target tree" we allow anything that is obtained from one of these tuples by descending operations and permutations only. If $g\in G$ is an arbitrary element, it follows from the fact that any two bases have a common descendant that $g=(Y_1,Y_2)$ where $Y_1$ and $Y_2$ are obtained in that way. Then, choose $X_i$ in our previously fixed set of tuples (what used to be the set of nodes in $\mathcal{T}$) of the same cardinality as $Y_1$ and $Y_2,$ and observe that $g=g_2g_1^{-1}$ with 
$g_1=(X_i,Y_1)$ and $g_2=(X_i,Y_2).$ These are precisely the type of elements we wanted to verify to be generators of the group.

The choice of that fixed set of tuples can be the same as  in Subsection \ref{gensys2}, but now we no longer need to construct the actual tree $\mathcal{T}$, we only need the nodes.  
For example, we can proceed as follows: as done before, fix a tuple $X_r$ with $r$ elements
and choose integers $m_1,\ldots,m_s$ with
$$d=\sum_{t=1}^sm_t(n_t-1).$$
There is a sequence of operations (first descending, then ascending) that we can perform on the last element of $X_r$ to get a new tuple with exactly $r+d$ elements that we denote $X_r\tau$.  We may repeat the process to get a new tuple $X_r\tau^2 $ and so on. We set $X_r\tau^0:=X_r$, let $X_{i+rd}:=X_r\tau^i$ for  $i\geq 0$ and take the obtained family as our prefixed set of \lq\lq sources". 

\noindent As seen above, our first set of generators is then

$$
\{(X_j,X_i\epsilon)\mid \epsilon\text{ sequence of descending operations}\}.
$$

\noindent Using Subsection \ref{gensys4} this can be further reduced to

$$
\{(X_j,X_i\epsilon)\mid \epsilon\text{ single strict descending operation or permutation}\}.
$$

Again, there is no serious need of the space $Z$ to see that this reduction is possible. One can just check that composition of these elements corresponds to composition of the associated descending operations, in a way similar to that of
\cite{burillocleary}.
The same  happens with the reduction performed in Subsection \ref{gensys5}: basically, we used $Z$ only to have some identities available that allowed us to eliminate some elements from our generating system, but all those identities can be easily checked by hand and one gets the same finite set in the end.
\end{remark}

\section{Finite presentations
\label{sec:finite-presentation}
}

\noindent In this section, we still assume that $U_r(\Sigma)$ is valid and bounded and we add the extra hypothesis that it is also complete to exhibit a procedure that gives a finite presentation.  To do that, we just replace $Z$ by a truncated version $Z^n$ and we use the results of Section \ref{sec:infinite-presentation} to obtain an explicit finite presentation.

\begin{definition}\label{complete}
Using the notation of Definition \ref{sigmadef}, suppose that for all $i\neq i'$, $i,i' \in S$ we have that $\Sigma_2^{i,i'} \neq \emptyset$ and that $f(j)=i'$ for all $j=1,...,n_i$ and $f'(j')=i$ for all $j'=1,...,n_{i'}$.
Then we say that $U_r(\Sigma)$  is {\sl complete}.
\end{definition}

\noindent Considering the Morse function $t(A)=|A|$ in $\mathcal{S}_r(\Sigma)$ 
we can filter the complex with respect to $t$, and define the truncated Stein complex

$$
\mathcal{S}_r(\Sigma)^{n}:=\mbox{ full subcomplex supported on } \{A\in\mathcal{S}_r(\Sigma)\mid t(A)\leq n\}.
$$

In particular this is just the simplicial complex $\mathcal{S}_r(\Sigma)^n$ obtained by considering bases of cardinality bounded by $n$ only. 
Note that this complex was used in \cite[Theorem 3.1]{francescobritaconcha} to show that under the conditions above, $V_r(\Sigma)$ is of type $F_\infty.$ The purpose of this section is to give a recipe for constructing explicit presentations.

Obviously, we can do the same with the complex $Z$ and consider its truncated version $Z^n$ where the tuples have at most $n$ elements. The map $u$ restricts to these truncated versions and the same argument as in Subsection \ref{eg} shows that there is a homotopy equivalence
$$u:Z^n\to|\mathcal{S}_r(\Sigma)^n|.$$
By \cite[Corollary 3.9]{fluch++} for the special case of $sV$ and \cite[Section 3]{francescobritaconcha} for the general case,  assuming that  $U_r(\Sigma)$ is valid, bounded and complete,  there is some positive integer $n_0$ depending on $\Sigma$, such that for any $n\geq n_0$ and any basis $B\in\mathcal{S}_r(\Sigma)$ with cardinality $|B|=n+1$  the link of $B$ in the Stein complex $\mathcal{S}_r(\Sigma)$ is   simply connected. Using Morse Theory (\cite[Corollary 2.6]{bestvinabrady}) we deduce that  for $n\geq n_0$ the inclusion $\mathcal{S}_r(\Sigma)^{n}\subseteq\mathcal{S}_r(\Sigma)^{n+1}$ induces an isomorphism in $\pi_1$ and $\pi_0$. As the space $\mathcal{S}_r(\Sigma)$ is contractible we have
$$\begin{aligned}
1=\pi_1(\mathcal{S}_r(\Sigma))&=\lim \pi_1(\mathcal{S}_r(\Sigma)^n),\\
1=\pi_0(\mathcal{S}_r(\Sigma))&=\lim \pi_0(\mathcal{S}_r(\Sigma)^n).
\end{aligned}$$
 and $1=\pi_1(\mathcal{S}_r(\Sigma)^n)=\pi_0(\mathcal{S}_r(\Sigma)^n)$ for $n\geq n_0$. From this we deduce that $\mathcal{S}_r(\Sigma)^n$ is path connected and simply connected for $n\geq n_0$. This, together with the fact that $u$ is a homotopy equivalence, implies that the same holds true for $Z^n$.
Finally, observe that  $Z^n$ being path connected implies that the same is true for $Z^n/G$. Therefore we can use $Z^n$ instead of $Z$ in Theorem \ref{gpresentation} and as $Z^n/G$ is finite we get a finite presentation. Hence we have

\begin{theorem}\label{fp} Let $U_r(\Sigma)$ be a valid, bounded and complete Cantor algebra, and let $n\geq 1$ such that $Z^n$ is simply connected. Then there is a finite presentation of $V_r(\Sigma)$  involving only strict generators $(i,k,t)$ with $i+n_t-1\leq n$, permutations $(i,\sigma)$ with $i\leq n$ and relators involving these generators only and which is obtained by truncating the presentation
$$\langle W_s\cup W_p\mid R_D\cup R_{\Sigma}\cup R_{PD}\cup R_{P}\rangle$$
given in
(\ref{infinitepresentation}).\end{theorem}

The main difference with the reduction process of Section \ref{sec:finite-generation} is that we are now also reducing the set of relators. Moreover, the \lq\lq truncated" set of  generators in the finite presentation obtained this way can be further reduced using the same arguments as in Section \ref{sec:finite-generation}.

\begin{example}\label{exV}
For $G=V$,  in  \cite[Corollary 3.9]{fluch++} there is a explicit condition on $n$ that implies that $Z^n$ is simply connected: we need  
$$1\leq \lfloor{n-1\over 3}\rfloor-1$$
thus we can take $n=7$. 
This means that the set of strict generators in Example \ref{infgenV}  can be reduced to $x_0,\ldots,x_4$ and the relators of disjoint type can be reduced to
$$x_k^{-1}x_lx_k=x_{l+1}$$
where $(k,l,l+1)$ is one of the following tuples:
$(0,1,2)$, 
$(0,2,3)$,
$(0,3,4)$,
$(1,2,3)$,
$(1,3,4)$,
$(2,3,4)$. At this point it is not difficult to write down a finite presentation of $V$.  Note also that in Example \ref{generatingV} we had already reduced to two strict generators $x_0$ and $x_1.$

Recently Bleak and Quick found a short finite presentation for $V$ with
$2$ generators and $9$ relations using different methods \cite{bleakquick}.

Using our methods we get a finite presentation of Thompson's group $F,$ and by using Tietze moves this presentation can be transformed to the well-known
$$\langle x_0,x_1\mid x_0^{-3}x_1x_0^{3}=x_1^{-1}x_0^{-2}x_1x_0^2x_1,x_0^{-2}x_1x_0^{2}=x_1^{-1}x_0^{-1}x_1x_0x_1\rangle$$
\end{example}

\begin{example}
For $G=sV$ we can also use \cite[Corollary 3.9]{fluch++} to compute the value of $n$ making $Z^n$ simply connected: we need 
$$1\leq \lfloor{n-1\over 2^s}\rfloor-1$$ thus we can take $n=1+2^{s+1}$.
Recall that when choosing the maximal tree in $Z/G$ we chose expansion by one colour only (see Subsection \ref{treeV}). Let that colour be denoted by $1$. For the same reason as in Example \ref{infgenV} we now have that elements of the form $(i,i,1)$ are the identity, and that for any colour $t$ and any $k<i-1$, we have that $(i,k,t)=(k+1, k, t)$. 

This now gives an infinite $W_s$, which for $G=2V$ can be listed as follows:
$$(i+1,i,1), \qquad (i+1,i,2), \qquad \mbox{and} \qquad (k,k,2),$$
which corresponds to the infinite order generators $A_{i-1},B_{i-1}$ and $C_{i}$ of Brin's infinite generating set of $2V$, see \cite{brin1} or \cite{burillocleary}.  Now by condition 1) of Theorem \ref{finitegeneration} this can be reduced to a finite generating set with $7$ strict generators, those where $i\leq 2$ and $k \leq 3,$ as well as a finite number of permutation generators.
Using Theorem \ref{fp} without any further reductions, we get a finite presentation where $i\leq 7$ and $k\leq 8.$

\end{example}

\section{Finite presentation for centralisers of finite subgroups}\label{presentation}

\noindent The proof of \cite[Theorem 4.9]{francescobritaconcha} can be used to show that whenever the
group $V_k(\Sigma)$ is finitely presented for any $k$, then so is $C_{V_r(\Sigma)}(Q)$
for any finite $Q\leq V_r(\Sigma),$ but the proof there does not yield an explicit finite
presentation. In this section we are going to construct a finite presentation of $C_{V_r(\Sigma)}(Q)$. To do that, we proceed as follows.
The first thing to observe is that, according to  \cite[Theorem 4.2]{francescobritaconcha} , the group $C_{V_r(\Sigma)}(Q)$ is a direct product of groups of the form
$$\colim(U_{r'}(\Sigma),L)\rtimes V_{r'}(\Sigma).$$
We now summarise the  notation developed in \cite{francescobritaconcha}. The semidirect product above is associated to a fixed transitive permutation representation $\varphi:Q\to S_m$ of the finite group $Q,$ where $S_m$ is the symmetric group of degree $m,$ the orbit length. Then $L$ is the centraliser of the image $\varphi(Q)$ in $S_m$ and thus  is a finite group. The number $r'$ depends on $\varphi,$ see \cite[Theorem 4.2]{francescobritaconcha},  but in order to simplify notation we will just set $r'=r$.
 The set of bases in $U_r(\Sigma)$ together with the expansion maps can be viewed as a directed graph and we let $(U_r(\Sigma),L)$ be the following diagram of groups associated to this graph: To each basis $A$ we associate $\text{Maps}(A,L)$,
the group with elements the maps from $A$ to $L$ where the group operation is induced by multiplication in $L$.
Each simple expansion $A\leq B$ corresponds to the diagonal map $\delta:\text{Maps}(A,L)\to\text{Maps}(B,L)$ with $\delta(f)(a\alpha_i^j)=f(a)$, where $a\in A$ is the expanded element. Then we consider the direct limit $\colim(U_r(\Sigma),L)$ whose elements are determined by some basis $A$ and a map $A\to L$. Observe that we may always assume that the basis $A$ satisfies $X_r\leq A$.

We begin by studying presentations for $\colim(U_r(\Sigma),L)$. We will obtain an infinite presentation (see Lemma \ref{prescolim} below) and then we will use the semidirect product action of $V_r(\Sigma)$ on this presentation together with the so called Burnside procedure described in the Appendix to get a (finite) presentation of the group $\colim(U_{r}(\Sigma),L)\rtimes V_{r}(\Sigma).$

We begin by constructing a generating system for the group $\colim(U_r(\Sigma),L)$. Take $x\in L$ and $A$ a basis with $X_r\leq A$. Take some subset $A_1\subseteq A$ and let $\chi_{A_1,x}\in\colim(U_r(\Sigma),L)$ be the element that maps every $a\in A_1$ to $x$ and every $a\in A\setminus A_1$ to the identity $1\in L$. It is easy to see that the set of all the elements of this form generates our group, but observe that there might be a uniqueness issue because if we had another basis $C$ with $A\leq C$ and $C_1$ is the subset of those elements in $C$ coming from elements in $A_1$, then $\chi_{A_1,x}=\chi_{C_1,x}$. To avoid this problem we set
$$\omega(A_1):=\{b\text{ descendant of elements in }X_r\mid aw=bw'\text{ for some $a\in A_1$ and descending words $w,w'$}\}
$$
(this was denoted $A_1(\L)$ in  \cite{francescobritaconcha}) and
$$\Omega:=\{\omega(A_1)\mid A_1\text{ subset of some basis }A\geq X_r\}.$$
At first sight, this set $\Omega$ seems different from the set $\Omega$ defined in \cite{francescobritaconcha}, which was defined for arbitrary finite subsets of the set of all descendants of elements in $X_r$. But \cite[Lemma 4.5 i)]{francescobritaconcha} shows that since we are assuming that our Cantor algebra is valid and bounded they are in fact equal.

We set $\chi_{\omega,x}:=\chi_{A_1,x}$, where  $\omega=\omega(A_1)$. Observe that the proof of \cite[Lemma 4.5 i)]{francescobritaconcha} also implies that $\omega(A_1)=\omega(C_1)$, provided that $A\leq C$ and $C_1$ is the subset of those elements in $C$ coming from elements in $A_1$ (or, in other words, $C_1=C\cap\omega(A_1)$). As a consequence one easily sees that for any $B_1$ subset of a basis $B$ with $X_r\leq B$
$$\chi_{A_1,x}=\chi_{B_1,x}\iff\omega(A_1)=\omega(B_1)$$
implying that $\chi_{\omega,x}$ is well defined.

We will need a bit more of the notation from \cite{francescobritaconcha}. Let $\omega\in\Omega$ and $A_1\subseteq A\geq X_r$ with $\omega=\omega(A_1)$.
We set
$$\|\omega\|=\Bigg\{\begin{aligned}
&t \text{ if }|A_1|\equiv t\text{ mod }d\text{ with }0<t\leq d\\
&0\text{ if }\omega=\emptyset.\\
\end{aligned}$$
This does not depend on $A_1$, see  \cite[Lemma 4.5 v)]{francescobritaconcha}. Now, let $\omega_1,\omega_2\in\Omega$ and 
 $A_1,A_2\subseteq A\geq X_r$ with $\omega_i=\omega(A_i)$ for $i=1,2.$ Observe that the fact that our Cantor algebra is bounded means that we can always find such $A_1$ and $A_2$. If $A_1\cap A_2=\emptyset$, we write $\omega_1\wedge\omega_2=\emptyset$. Again, this is well defined by 
  \cite[Lemma 4.5 vi)]{francescobritaconcha}.

\begin{lemma}\label{prescolim}  The following is a presentation of
$\colim(U_r(\Sigma),L)$:
$$\langle (\chi_{\omega,x})_{\omega\in\Omega\setminus\emptyset,x\in L}\mid \mathcal{R}_1,\mathcal{R}_2,\mathcal{R}_3\rangle$$
where
$$\mathcal{R}_1=\{\chi_{\omega,xy}^{-1}\chi_{\omega,x}\chi_{\omega,y}\mid\omega\in\Omega,x,y\in L\},$$
$$\mathcal{R}_2=\{[\chi_{\omega,x},\chi_{\omega',y}]\mid\omega,\omega'\in\Omega,\omega\wedge\omega'=\emptyset\}\text{ and}$$
$$\mathcal{R}_3=\{\chi_{\omega,x}^{-1}\chi_{\omega_1,x}\chi_{\omega_2,x}\mid\omega,\omega_1,\omega_2\in\Omega,\omega=\omega_1\mathrel{\dot\cup}\omega_2\},$$
where $\omega_1\mathrel{\dot\cup}\omega_2$ denotes
the disjoint union.
Moreover $V_r(\Sigma)$ acts  by permutations with finitely many orbits on this presentation.
\end{lemma}

\begin{proof}  As observed above,
any $\chi\in\colim(U_r(\Sigma),L)$ is a product of elements of the form $\chi_{\omega,x}$ for a
suitable $\omega\in\Omega$ and $x\in L$.
Let $F$ denote the free group on the set $\{\widetilde\chi_{\omega,x}\mid \omega\in\Omega\setminus\emptyset,x\in L\}$. There is   an epimorphism
$$F\buildrel\tau\over\twoheadrightarrow\colim(U_r(\Sigma),L)$$
with $\tau(\widetilde{\chi}_{\omega,x})=\chi_{\omega,x}$.
Let $G$ be the abstract group defined in the statement of the
result for the generators $\widetilde{\chi}_{\omega,x}$. It is immediate to verify that the epimorphism $\tau$
defined above induces an epimorphism from $G$ to $\colim(U_r(\Sigma),L)$
which we still call $\tau$. 
This follows since all  relations inside $G$ are easily verified to hold for the images
$\tau(\widetilde{\chi}_{\omega,x})$. Assume that we have a word $\widetilde{w}=w(\widetilde{\chi}_{\omega_1,x_1}, \ldots, \widetilde{\chi}_{\omega_k,x_k})$, for some $\omega_1,\ldots,\omega_k\in\Omega$ and $x_1,\ldots,x_k\in L$.
Assume further that
\[
1=\tau(\widetilde{w})=\tau(w(\widetilde{\chi}_{\omega_1,x_1}, \ldots, \widetilde{\chi}_{\omega_k,x_k}))=w(\tau(\widetilde{\chi}_{\omega_1,x_1}), \ldots, \tau(\widetilde{\chi}_{\omega_k,x_k})).
\]

Let $X_r\leq A$ be a basis with subsets $A_i\subseteq A$ such that $\omega_i=A_i(\L)$ for $i=1, \ldots, k$. 
We now refine the set $\{A_1,...,A_k\}$ to a set $\{A'_1,...,A'_{k'}\}$ of subsets of $A$ such that for all $i,j \leq k'$ either $A'_i\cap A'_j=\emptyset$ or $A'_i=A'_j.$  By suitably applying the relations in $\mathcal{R}_3$ to both the original word $w(\widetilde{\chi}_{\omega_1,x_1}, \ldots, \widetilde{\chi}_{\omega_k,x_k})$  and its image
$w:=\tau(\widetilde{w})=w({\chi}_{\omega_1,x_1}, \ldots,{\chi}_{\omega_k,x_k}),$ we may rewrite 
each occurrence of $\chi_{\omega_i,x_i}$ and $\widetilde{\chi}_{\omega_i,x_i}$ in terms of suitable new elements
$\tau(\widetilde{\chi}_{\omega'_j,y_j})$ and ${\chi}_{\omega'_j,y_j}$ for $1\leq j\leq k'$, so that either $\omega'_j\wedge\omega_i'=\emptyset$ or $\omega'_j=\omega'_i$. 

Reordering them so that $\omega_1,\ldots,\omega_u$ for $1\leq u\leq k'$ are pairwise distinct and applying the relations in $\mathcal{R}_2$ and $\mathcal{R}_1$ 
to group together the suitable products of the $y_j$'s
we obtain new words

\[
\widetilde{w}\sim\widetilde{w}'=\widetilde{\chi}_{\omega'_1,z_1} \ldots \widetilde{\chi}_{\omega'_{u},z_{u}}, \qquad
{w}\sim {w}'={\chi}_{\omega'_{1},z_{1}} \ldots {\chi}_{\omega'_{u},z_{u}},
\]
where the $\omega_i'$'s are pairwise disjoint.

If ${w}'\sim 1$, we must have  $z_i=1$ for any $1\leq i\leq u$,
by applying the word ${w}'$ to an $a \in A_i$ such that
$A_i(\mathcal{L})=\omega_i'$.
From $\mathcal{R}_1$ it is immediate to see that $\widetilde{\chi}_{\omega,1}=1$ for any $\omega\in\Omega$
 so we also have $\widetilde{w}\sim \widetilde{w}'\sim 1$
and  $G$ gives a presentation of $\colim(U_r(\Sigma),L)$.

By \cite[Lemma 4.7]{francescobritaconcha}, the group $V_r(\Sigma)$ acts by permutations on $\Omega$. Moreover, for any $g\in V_r(\Sigma)$, if $\omega,\omega'\in\Omega$ are such that $\omega\wedge\omega'=\emptyset$, then $g\omega\wedge g\omega'=\emptyset$ and if $\omega=\omega_1\cup\omega_2$ for $\omega_1,\omega_2\in\Omega$, then $g\omega=g\omega_1\cup g\omega_2$. Therefore $V_r(\Sigma)$ acts by permutations on this presentation.
To prove the last statement, it suffices to check the following:

\smallskip
\noindent \emph{Claim 1}: The set of generators is $V_r(\Sigma)$-finite.

\smallskip
\noindent \emph{Claim 2}: Each of the sets of relations $\mathcal{R}_1, \mathcal{R}_2,\mathcal{R}_3$ is $V_r(\Sigma)$-finite.

\smallskip
As the group $L$ is finite, both claims follow from slight variations of the proof of 
\cite[Lemma 4.7]{francescobritaconcha}.
For example, for Claim 2 for $\mathcal{R}_2$, it suffices to check that
whenever we have $\omega,\omega',\widehat\omega,\widehat\omega'\in\Omega$ with $\omega\wedge\omega'=\emptyset$, $\widehat\omega\wedge\widehat\omega'=\emptyset$, $\|\omega\|=\|\widehat\omega\|$ and $\|\omega'\|=\|\widehat\omega'\|$, then there is some $g\in V_r(\Sigma)$ such that for any $x\in L$, $\chi_{\widehat\omega,x}=\chi_{\omega,x}^g$ and $\chi_{\widehat\omega',x}=\chi_{\omega',x}^g$. To get  a suitable $g,$  choose bases $X_r\leq A,\widehat A$ so that for  $B,B'\subseteq A$ and $\widehat B,\widehat B'\subseteq \widehat A$, we have $\omega=\omega(B)$, $\omega'=\omega(B')$, $\widehat\omega=\omega(\widehat B)$, $\widehat\omega'=\omega(\widehat B')$, $|A|=|\widehat A|$, $|B|=|\widehat B|$ and $|B'|=|\widehat B'|$. The assumptions imply that $B\cap B'=\emptyset=\widehat B\cap\widehat B'$. So we may choose a $g\in V_r(\Sigma)$ with $gA=\widehat A$, $gB=\widehat B$ and $gB'=\widehat B'$.

In a completely analogous way one proves that for $\omega,\omega_1,\omega_2,\widehat\omega,\widehat\omega_1,\widehat\omega_2\in\Omega$ with $\omega=\omega_1\cup\omega_2$,
$\widehat\omega=\widehat\omega_1\cup\widehat\omega_2$, $\|\omega\|=\|\widehat\omega\|$, $\|\omega_1\|=\|\widehat\omega_1\|$ and $\|\omega_2\|=\|\widehat\omega_2\|$ there is some $g\in V_r(\Sigma)$ such that for any $x\in L$, $\chi_{\widehat\omega,x}=\chi_{\omega,x}^g$, $\chi_{\widehat\omega_1,x}=\chi_{\omega_1,x}^g$ and $\chi_{\widehat\omega_2,x}=\chi_{\omega_2,x}^g$.

\end{proof}

\begin{proposition}\label{finpres} Assume that the group $V_r(\Sigma)$ is finitely presented. Let $Q\leq V_r(\Sigma)$ be a finite subgroup. Given a finite presentation of $V_r(\Sigma),$ Lemma \ref{prescolim} together with Theorem \ref{thm:burnside-3} yield an explicit finite presentation of $C_{V_r(\Sigma)}(Q).$ 
\end{proposition}

\begin{proof} By \cite[Theorem 4.2]{francescobritaconcha}, it suffices to construct an explicit finite presentation of a group of the form $H=\colim(U_r(\Sigma),L)\rtimes V_r(\Sigma)$ when $L$ is an arbitrary finite group.
Let $V_r(\Sigma) =\langle Z \mid T\rangle$
be a finite presentation of $V_r(\Sigma)$ and let $\colim(U_r(\Sigma),L)=\langle Y\mid R\rangle$ be the presentation constructed in Lemma \ref{prescolim}. We need to verify the hypotheses of Theorem \ref{thm:burnside-3}.
In Lemma \ref{prescolim} we have already checked that the group $V_r(\Sigma)$ acts by permutations in this presentation and that there are only finitely many orbits under that action. We may therefore choose $Y_0\subseteq Y$ and $R_0\subseteq R$ finite sets of representatives of these orbits.

The argument in Section \ref{ssec:burnside-2} thus implies that the group $G$
has the following presentation:
\begin{align*}
\langle Y_0,Z \mid \widehat R_0,T, [\mathrm{Stab}_{V_r(\Sigma)}(y),
y],y\in Y_0 \rangle.
\end{align*}

We can give explicit descriptions of possible choices for the sets $Y_0, R_0$. Set $X_r=\{x_1,\ldots,x_r\}$ and let $\omega_i=\omega(\{x_1,\ldots,x_i\})$ for $i=1,\ldots,r$. Then:
$$Y_0=\{\chi_{\omega_i,z}\mid 1\leq i\leq r,z\in L\}.$$

To describe $R_0$, we are going to split it into three pairwise disjoint subsets $R_0=R_0^1\cup R_0^2\cup R_0^3$, according to the three subsets of relations $\mathcal{R}_1$, $\mathcal{R}_2$ and $\mathcal{R}_3$ of Lemma \ref{prescolim}. The simplest one is $R_0^1$:
$$R_0^1=\{\chi_{\omega_i,zy}^{-1}\chi_{\omega_i,z}\chi_{\omega_i,y}\mid 1\leq i\leq r,z\in L\}.$$

For $R_0^2,R_0^3$ it is more convenient to fix a basis $X_r\leq A$ with $|A|\geq 2r$. Then we may choose

$$R_0^2=\{[\chi_{\omega,z},\chi_{\omega',z}]\mid z\in L,\omega=\omega(A_1),\omega'=\omega(A_1'),A_1,A_1'\subseteq A,A_1\cap A_1'=\emptyset\},$$
$$R_0^3=\{\chi_{\omega,z}^{-1}\chi_{\omega_1,z}\chi_{\omega_2,z}]\mid z\in L,\omega_1=\omega(A_1),\omega_2=\omega(A_2),\omega=\omega_1\mathrel{\dot\cup} \omega_2,A_1,A_2\subseteq A\}.$$
Observe that these choices of $R_0^2$ and $R_0^3$ yield redundant presentations.

The previous presentation may not be finite because of all the relations needed to
form $[\mathrm{Stab}_{V_r(\Sigma)}(y), y]$ where $y\in Y_0$.
Notice that $g \in \mathrm{Stab}_{V_r(\Sigma)}(y)$ if and only if
$g(\omega)=\omega$ where $y=\chi_{\omega,z}$ for some $z\in L$.
By \cite[Lemma 4.7]{francescobritaconcha} and the assumption on $V_r(\Sigma)$
we deduce that $\mathrm{Stab}_{V_r(\Sigma)}(y)$ is finitely generated by
some generators $\mu_1,\ldots, \mu_m$.

Consider now the following $m$ relations, which are a subset of the 
stabiliser relations $[\mathrm{Stab}_{V_r(\Sigma)}(y), y]$:
\begin{eqnarray}
\label{rel4} \mu_i \chi_{\omega,z} \mu_i^{-1} = \chi_{\omega,z}, \qquad i=1,\ldots, m.
\end{eqnarray}
If $g \in \mathrm{Stab}_{V_r(\Sigma)}(y)$, then $g=w(\mu_1,\ldots,\mu_m)$
and the stabiliser relation $g \chi_{\omega,z} g^{-1}=\chi_{\omega,z}$ is thus obtained by starting from
relation (\ref{rel4}) for some $i$ and then suitably conjugating this relation to build the word $w$.

Therefore, by Lemmas \ref{thm:burnside-1} and \ref{thm:burnside-2},
the group $H$ has the following finite presentation:
\begin{align*}
\langle Y_0, Z \mid \widehat R_0,T, [\mu_i, y], i=1,\ldots, m, y\in Y_0 \rangle,
\end{align*}
where the elements $\mu_1,\ldots, \mu_m$ are expressed as words in the generators
$Z$.
\end{proof}

\appendix
\section{The Burnside procedure}

We shall now give an outline of the Burnside procedure used in the proof of Proposition \ref{finpres}. As mentioned in the Introduction, we do not claim any originality for this. For example, this procedure has been used, without proof, in 
\cite{kassabovetc}. We are not aware of any place where a proof is presented. 
Hence we include it here for completeness.

\bigskip
The goal is to find a small finite presentation of a group, in the cases where the following
procedure can be applied. The idea is to look for a possibly infinite, but well behaved,
presentation of a group $G$ and a group $Q$ such that the action
of $Q$ on the generators and relators of $G$ cuts them down to a very small number.
At a later stage, the group $Q$ will be assumed to be a subgroup of $G$ and
its action will return a new smaller presentation.

\subsection{Preliminary lemmas}
\label{ssec:burnside-1}

The beginning of this procedure is general and we only require each of the groups
$G$ and $Q$ to have a presentation, without any assumption on them.

Let $G=\langle Y \mid R \rangle$ and $Q=\langle Z \mid T \rangle$ be groups. Let $Q$
act on $Y$
by permutations. Notice that $R \le F(Y)$, the free
group generated by $Y$, and observe that $Q$ also acts on $F(Y)$.
We assume that $Q(R) = R$.
Let $Y_0$ be a set of representatives for the $Q$-orbits in $Y$ and $R_0$ be a set of representatives for the $Q$-orbits in $R$.
We observe that $R_0 \le \langle \, t(a_0) \mid a_0 \in Y_0, t \in Q \, \rangle$ that is, we may express the elements of $R_0$
as products of the results of $Q$ acting on elements of $Y_0$.
In the special case that $Q$ is a subgroup of $G$, we will be able to
express elements in $R_0$ as products of conjugates of elements in $Y_0$ by elements in $Q$.
Hence each element of $R_0$, seen as an element in $G$, can be written in more than one way and we fix an expression of the type $t_1(a_1) \ldots t_k(a_k)$ for such element.
We then define the set  $\widehat{R}_0 \subset \langle t a_0 t^{-1} \mid
a_0 \in Y_0, t \in Q \rangle$ to be the set of fixed expressions for the elements of $R_0$,
where we have replaced the action of $Q$ on $Y_0$ by the conjugation of elements.
That is, if $t_1(a_1) \ldots t_k(a_k)$ is a fixed expression in $R_0$, the corresponding element in $\widehat R_0$ 
is $t_1 a_1 t_1^{-1} \ldots t_k a_k t_k^{-1}$. The set $\widehat{R}_0$ is thus a set of formal expressions which will be used later
to express relations in the groups.

\begin{lemma}
\label{thm:burnside-1}
Following the notation previously defined, we have
\[
G \rtimes Q \cong \langle Y_0,Z \mid \widehat{R}_0, T, [\mathrm{Stab}_Q(y),y], y \in Y_0 \rangle,
\]
where the semi-direct product is given by the action of $Q$ on $G$ as follows: for all $g_1,g_2 \in G$ and $t_1,t_2 \in Q,$ multiplication is given by
$(g_1,t_1)(g_2,t_2)=(g_1 \cdot t_1(g_2),t_1t_2).$
\end{lemma}

\begin{proof}
Let $H$ be the group presented by
$\langle Y_0,Z \mid \widehat{R}_0, T, [\mathrm{Stab}_Q(y),y], y \in Y_0 \rangle$.
Define the group homomorphism $\varphi:F(Y_0 \cup Z) \to G \rtimes Q$ by sending $a_0 \in Y_0$ to $(a_0,1) \in G \rtimes Q$ and $c \in Z$ to $(1,c) \in G \rtimes Q$.
By construction we see that 
\[\varphi(t) \varphi(a_0) \varphi(t)^{-1}=(t(a_0),1)\tag{$\ast$}\]
 for any word $t\in Q$.

\medskip\noindent \emph{Claim 1.} The map $\varphi$ induces a homomorphism $H \to G \rtimes Q$, which we still
call $\varphi$.

\begin{proof}[Proof of Claim 1]
If $d \in T$ is a relation in $H$, then $d=c_1 \ldots c_k$, for some $c_i \in Z$, and
$\varphi(c_1) \ldots \varphi(c_k)=(1,1)$. 
Let now $\widehat{b}_0 \in \widehat{R}_0$ be a relation in $H$, then
$\widehat{b}_0=t_1 a_1 t_1^{-1} \ldots t_k a_k t_k^{-1}$
for some $a_i \in Y_0$ and $t_i \in Q$. Moreover, by applying $(\ast)$ above, we get

\begin{eqnarray*}
\prod_{i=1}^k \varphi(t_i)\varphi(a_i)\varphi(t_i)^{-1} = (\prod_{i=1}^k t_i(a_i),1) =(1,1).
\end{eqnarray*}

\noindent Finally let $a_0 \in Y_0$,
$t \in \mathrm{Stab}_Q(a_0)$. Thus we have, using $(\ast)$ again:
\begin{eqnarray*}
\varphi(t)\varphi(a_0)\varphi(t)^{-1}\varphi(a_0)^{-1} &=& (t(a_0),1)(a_0^{-1},1) = (1,1).
\end{eqnarray*}
Now we just apply Von Dyck's theorem.
\end{proof}

\medskip\noindent \emph{Claim 2.} The map $\varphi$ is surjective.

\begin{proof}[Proof of Claim 2]
Any element $(1,t) \in \{1\} \times Q:=\{(1,s) \mid s \in Q \}$
can be written as $(1,c_1 \ldots c_k)$
for suitable $c_i \in Z$ and so $\varphi(H)$ contains $\{1\} \times Q$.
We observe that any element of $G \times \{1\}:=\{(h,1) \mid h \in G \}$ can be
written as $(t_1(a_1) \ldots t_k(a_k),1)$ for suitable $a_i \in Y_0$ and $t_i \in Q$. By arguing as in  Claim 1 we have
$(g,1)= \varphi(\prod_{i=1}^k t_i a_i t_i^{-1})$.
Thus, $\varphi(H) \ge \langle G \times \{1\}, \{1\} \times Q \rangle = G \rtimes Q$.
\end{proof}

\medskip\noindent \emph{Claim 3.} The map $\varphi$ is injective.

\begin{proof}[Proof of Claim 3]
Any element of $Y$ can be written as
$t(a_0)$, for some $a_0 \in Y_0$ and $t \in Q$. Define
$\overline{Y}^*=\{t a_0 t^{-1} \mid a_0 \in Y_0, t \in Q\}$
to be the set of symbols of $Y$ where we have replaced the action
of $Q$ with the conjugation of elements. We notice that, if $t(a_0)=s(a_0)$, then
$t^{-1}s \in \mathrm{Stab}_Q(a_0)$ and we thus define an equivalence relation on $\overline{Y}^*$
by writing $t a_0 t^{-1} \sim s a_0 s^{-1}$ if and only if $t^{-1}s \in \mathrm{Stab}_Q(a_0)$.
We define $\overline{Y}:=\overline{Y}^*/\sim$ the collection of equivalence classes.

If $a \in Y$ and $a=t(a_0)$, for some $a_0 \in Y_0$ and $t \in Q$, we define an element $\overline{a}$
of $\overline{Y}$ by setting $\overline{a}=\{s a_0 s^{-1} \mid t^{-1}s \in \mathrm{Stab}_Q(a_0)\}$.
With this notation, we observe that $Q$ acts on $\overline{Y}$ through
\[
(s,\overline{a}) \to s \cdot \overline{a} := \overline{st a_0 t^{-1} s^{-1}},
\]
for some $a_0 \in Y_0, t\in Q$ such that $\overline{a}=\overline{t a_0 t^{-1}}$. Also, notice
that the map $\psi:Y \to \overline{Y}$ sending $a \mapsto \overline{a}$
is a $Q$-equivariant bijection, that is  $\psi(s a)=s \psi(a)= s \cdot \overline{a}$
for all $s \in Q$. Hence
the action of $Q$ on $Y$ is equivalent to the action of $Q$ on $\overline{Y}$.
For each element $\overline{a} \in \overline{Y}$ we can fix a representative $t a_0 t^{-1} \in F(Y_0 \cup Z)$
and we call the set of representatives $\widehat{Y}$. By construction, every element $\widehat{b}_0 \in \widehat{R}_0$
can be uniquely written as $\widehat{b}_0=t_1 a_1 t_1^{-1} \ldots t_k a_k t_k^{-1}$,
so we define $\overline{R}_0 \subseteq F(\overline{Y})$ be the set of elements $\overline{t_1 a_1 t_1^{-1}}
\ldots \overline{t_k a_k t_k^{-1}}$. We then let
$\overline{R}
\subseteq F(\overline{Y})$ to be the
set of all elements $\overline{t t_1 a_1 t_1^{-1} t^{-1}}  \ldots \overline{ t t_k a_k t_k^{-1}t^{-1}}$, for any $t \in Q$.

With these definitions, it makes sense to say that the normal closure 
$F(\overline{R})^{F(\overline{Y})}$ inside
$F(\overline{Y})$ is isomorphic to $F(R)^{F(Y)}$ inside $F(Y)$. Also notice that
$F(\overline{Y}) \cong F(\overline{Y}^*/\sim) = \langle \overline{Y}^* \mid R_\sim\rangle$, where $R_\sim$
is the set of all relations of the type $t a_0 t^{-1} \sim s a_0 s^{-1}$ if and only if $t^{-1}s \in \mathrm{Stab}_Q(a_0)$. 

Let $w \in H$, such that $\varphi(w)=(1,1)$. Let $w= c_1 a_1 c_2 a_2 \ldots a_k c_{k+1}$ for $a_i \in Y_0$ and $c_i \in \langle Z \rangle$ and we rewrite $w$ as
\[
w=(c_1 a_1 c_1^{-1}) (c_1 c_2 a_2 c_2^{-1} c_1^{-1}) \ldots (c_1 c_2 \ldots c_k a_k c_k^{-1} \ldots c_1^{-1})
c_1 c_2 \ldots c_k c_{k+1}
\]
Define $t_i= c_1 \ldots c_i$. Then, up to replacing $t_i$ with another suitable $t_i' \in Q$, we can assume that
$\overline{t_i a_i t_i^{-1}} \in \widehat{Y}$. Hence we can write
$w=\left( t_1 a_1 t_1^{-1} \ldots t_k a_k t_k^{-1} \right) t_{k+1}$ and, applying $\varphi$
to the rewriting of $w$ we get $(1,1)
=(t_1(a_1) \ldots t_k(a_k),t_{k+1}).$

Since $t_{k+1}=1$ inside $Q$, we can use the relations of $Q$ to rewrite $t_{k+1}=1$ inside $H$.
Similarly, since $t_1(a_1) \ldots t_k(a_k)=1$ inside $G$ and since the normal closure
$F(\overline{R})^{F(\overline{Y})}$ inside $F(\overline{Y})$ is isomorphic to $F(R)^{F(Y)}$ inside $F(Y)$,
we can use the relations of $G$ to rewrite $t_1 a_1 t_1^{-1} \ldots t_k a_k t_k^{-1}=1$ inside $H$. Therefore $w=1$ in $H$ and so $\varphi$ is injective.
\end{proof}

\medskip\noindent The map $\varphi$ is thus a group isomorphism and we are done.
\end{proof}

\noindent The following result does not depend on the presentations of the relevant groups and relies 
only on the definition of semidirect product.

\begin{lemma}
\label{thm:burnside-2}
Let $G$ be a group and $Q \le G$. Let $G \rtimes Q$
be the semidirect product constructed using the action of $Q$ on $G$
by conjugation inside $G$. Then
\[
G \rtimes Q \cong G \times Q.
\]
\end{lemma}
\begin{proof}
Let $H:=G \rtimes Q$ with product given by
$(a,x)(b,y)=(axbx^{-1},xy)$. It is clear that $\widetilde{Q}=\{(t^{-1},t) \mid t \in Q \}$ is a subgroup 
of $H$ and $\widetilde{Q} \cong Q$. Since $(a,x)=(ax,1)(x^{-1},x)$, then $H$ is generated by $G \times \{1\}$ and $\widetilde{Q}$. It is straightforward to verify that $Q$ is normal and so, since
$G \times \{1\}$ too, we get $G \rtimes Q \cong (G \times \{1\}) \times \widetilde{Q} \cong G \times Q$.
\end{proof}

\subsection{The Burnside procedure}
\label{ssec:burnside-2}

We are now ready to explain the Burnside procedure.
We make two additional assumptions with respect 
to those in Subsection \ref{ssec:burnside-1}. We assume that:
\begin{enumerate}
\item the presentation $Q=\langle Z \mid T \rangle$ is finite,
\item  the number of $Q$-orbits in $Y$ is finite (and possibly very small,
in practical applications),
\item the number of $Q$-orbits in $R$ is finite (and also possibly very small),
\item the stabilisers $\mathrm{Stab}_Q(y)$ are finitely generated, for $y \in Y_0$.
\end{enumerate}
Let  $G$ and $Q$ be as defined in Lemma \ref{thm:burnside-1},
 $Q \le G$ and let $Q$ act by conjugation on $G$, then Lemmas 
\ref{thm:burnside-1} and \ref{thm:burnside-2}
imply that
\[
G \times Q \cong \left\langle Y_0,Z \mid \widehat{R}_0, T, [\mathrm{Stab}_Q(y),y] \text{ for }
y \in Y_0  \right\rangle.
\]
We rewrite $Z$ in terms of $Y_0$ and then mod out $Q$. We also
use the finite generation of $\mathrm{Stab}_Q(y)$ to rewrite the stabiliser relations
as conjugations. Therefore we obtain:

\begin{theorem}[Burnside procedure]
\label{thm:burnside-3}
Let $G,Q$ be the groups defined in Lemma \ref{thm:burnside-1}.
Assume that 
\begin{enumerate}
\item $Q \le G$ and $Q$ acts by conjugation on $G$,
\item $Q=\langle Z \mid T \rangle$ is finitely presented,
\item the number of $Q$-orbits in $Y$ is finite,
\item the number of $Q$-orbits in $R$ is finite,
\item the stabilisers $\mathrm{Stab}_Q(y)$ are finitely generated, for $y \in Y_0$.
\end{enumerate}
Then there exists a finite presentation of $G$ of the type:
\[
G = \left\langle 
\begin{array}{cc}
 & R_0, T, \\
Y_0, Z \; \; \Big \vert & cyc^{-1}=y, \text{ for } y \in Y_0, c \text{ generator of } \mathrm{Stab}_Q(y), \\
& \text{finitely 
many extra relations} 
\end{array}
\right\rangle,
\]
where the extra relations are obtained in the following way:
there is a relation for every element $c \in Z$ and
it has the form 
\[
c=\text{word in conjugates of 
elements of } Y_0 \text{ by elements of }Z.
\]
\end{theorem}

\subsection{An application}
\label{ssec:burnside-3}

The following example is taken from \cite{kassabovetc}. Recall the following
presentation for the alternating group
\[
\mathrm{Alt}(n+2)=\left\langle x_1,\ldots,x_p \mid (x_i)^3, (x_i x_j)^2, i \ne j \right\rangle
\]
where $x_i$ can be realised as the $3$-cycle $(i \; \; \; \; n+1 \; \; \; \; n+2)$. Hence
\[
\mathrm{Alt}(7) = 
\left\langle x_1,x_2,x_3,x_4,x_5 \mid (x_i)^3, (x_i x_j)^2, i \ne j \right \rangle := G.
\]
On the other hand, it can be shown that
\[
\mathrm{Alt}(5) = \left\langle a,b \mid a^5, b^2, (ab)^3 \right\rangle := Q,
\]
where $a$ can be realised as $(1 \; 2 \; 3 \; 4 \; 5)$ and $b=(2 \; 3)(4 \; 5)$. Let $z:=x_1=(1 \; 6 \; 7)$ and observe that $x_i=z^{a^{i-1}}$, for $i=1,\ldots,5.$  Now we check that 
\[
\begin{array}{l}
Y=\{x_1,\ldots,x_5\}, \\
Y_0=\{ z\}, \\
R=\{(x_i)^3, (x_i x_j)^2, i \ne j\},\\
R_0=\{z^3, \left(z z^a \right)^2\}, \\
Z=\{a,b\}, \\
T=\{a^5,b^2,(ab)^3\},
\end{array}
\]
satisfies the conditions of Corollary \ref{thm:burnside-3}. Noting that
$\{[\mbox{Stab}_Q(y),y] \text{ for } y \in Y_0 \} = \{\left[z,b \right], \left[z,(ba)^a\right]\}$, we have that 
\[
G \times Q = \left\langle a,b,z \mid a^5, b^2, (ab)^3, z^3, \left(z z^a \right)^2,
\left[z,b \right], \left[z,(ba)^a\right] \right \rangle
\]

We can write
$a=w_1(x_1,\ldots,x_5)$ and $b=w_2(x_1,\ldots,x_5)$, for suitable words
$w_1,w_2 \in F(x_1,\ldots,x_5)$ and then Corollary \ref{thm:burnside-3}
yields the following finite presentation for $\mathrm{Alt}(7)$:
\[
\mathrm{Alt}(7) = \left\langle
a,b,z \mid R_0, T, \left[z,b \right], \left[z,(ba)^a\right],
a^{-1}w_1(z,z^a,\ldots,z^{a^4}), b^{-1}w_2(z,z^a,\ldots,z^{a^4})
\right\rangle.
\]

\end{document}